\documentclass[final,hidelinks,onefignum,onetabnum]{siamart220329}

\usepackage{lipsum}
\usepackage{amsfonts}
\usepackage{graphicx}
\usepackage{epstopdf}
\usepackage{algorithmic}
\ifpdf
  \DeclareGraphicsExtensions{.eps,.pdf,.png,.jpg}
\else
  \DeclareGraphicsExtensions{.eps}
\fi

\newsiamremark{remark}{Remark}
\newsiamremark{hypothesis}{Hypothesis}
\crefname{hypothesis}{Hypothesis}{Hypotheses}
\newsiamthm{claim}{Claim}

\headers{Low-rank Tensor Train Decomposition Using TensorSketch}{Z. Chen, H. Jiang, G. Yu, and L. Qi}

\title{Low-rank Tensor Train Decomposition Using TensorSketch\thanks{
This work was partially supported by National Natural Science Foundation of China (No. 12071104) and Natural Science Foundation of Zhejiang Province (No. LD19A010002, No. LY22A010012).}}

\author{Zhongming Chen\thanks{Department of Mathematics, School of Sciences, Hangzhou Dianzi University, Hangzhou 310018, China
  (\email{zmchen@hdu.edu.cn}, \email{hlin\_j@163.com}, \email{maghyu@163.com}).}
\and Huilin Jiang\footnotemark[2]
\and Gaohang Yu\footnotemark[2]
\and Liqun Qi\footnotemark[2] \thanks{Department of Applied Mathematics, The Hong Kong Polytechnic University, Hung Hom, Kowloon, Hong Kong (maqilq@polyu.edu.hk).}
}

\ifpdf
\hypersetup{
  pdftitle={Low-rank Tensor Train Decomposition Using TensorSketch},
  pdfauthor={Z. Chen, H. Jiang, G. Yu, and L. Qi}
}
\fi

\begin{document}

\maketitle

\begin{abstract}
Tensor train decomposition is one of the most powerful approaches for processing high-dimensional data. For low-rank tensor train decomposition of large tensors, the alternating least squares (ALS) algorithm is widely used by updating each core tensor alternatively. However, it may suffer from the curse of dimensionality due to the large scale of subproblems. In this paper, a novel randomized proximal ALS algorithm is proposed for low-rank tensor train decomposition by using TensorSketch, which allows for efficient implementation via fast Fourier transform. The theoretical lower bounds of sketch size are estimated for approximating the optimal value of subproblems. Numerical experiments on synthetic and real-world data also demonstrate the effectiveness and efficiency of the proposed algorithm.
\end{abstract}



\section{Introduction}
\label{sec_1}
Tensors are multi-dimensional arrays and generalizations of matrices to higher orders, which could be regarded as natural representations of large-scale data arising from chemometrics, statistics, data science, etc. In practical applications, one of the important tasks is to mine the low-dimensional structure hidden behind the tensors. Tensor decompositions \cite{KB09} are powerful tools for compressing, approximating, as well as extracting important features from high-dimensional data, and are widely used in signal processing \cite{DD98}, data mining \cite{KS08}, computer vision \cite{PKC21} and machine learning \cite{AGH14, CBS17}. The main tensor decompositions include CP decomposition \cite{H27}, Tucker decomposition \cite{T66}, tensor train (TT) decomposition \cite{O11}, tenor ring decomposition \cite{ZZX16} and so on. CP decomposition provides a useful way to factorize a tensor into the sum of rank-1 tensors. Unfortunately, it is not reliable due to the difficulty of determining the number of rank-1 components. Tucker decomposition is more stable than CP decomposition, but it suffers from the curse of dimensionality. On the other hand, TT decomposition is not affected by the curse of dimensionality and is more reliable. In this paper, we mainly focus on TT decomposition which is becoming increasingly popular due to its stability and efficiency. 

The tensor train decomposition can decompose a large tensor into the product of a series of third-order tensors. One direct way to compute low-rank TT decomposition is called TT-SVD \cite{O11}, which is based on the truncated singular value decomposition (SVD) of auxiliary unfolding matrices. Another widely used method based on optimization is called TT-ALS \cite{HRS12}, which updates each core tensor alternatively by solving corresponding least squares problem. However, both methods may suffer from the curse of dimensionality. In other words, the computation cost of both methods goes exponentially with the order of tensors, which is impractical for large-scale problems. As datasets grow larger and larger, there is an increasing need for methods to handle them. One possible solution to the challenge is the use of randomization, which has proven to be effective in computing the low-rank approximations of large-scale matrices \cite{HMT11, M11, TYUC17, BYDW18, MT20}.

In the realm of tensor decomposition, different randomized techniques have been applied to accelerate the low-rank approximations of tensors \cite{BBK18, EMBK20, MSK20, CWY20, CWY21, AAA21, DYQC23}. For TT decomposition, Huber et al. proposed randomized TT decomposition which is a robust alternative to the classical deterministic TT-SVD algorithm at low computational expenses \cite{HSW17}. Che et al. proposed an adaptive randomized algorithm for computing the tensor train approximations of tensors \cite{CW19}. To make full use of TT format, Shi et al. proposed parallelizable sketching algorithms that compute the low-rank TT decomposition from various tensor inputs \cite{SRT23}. Yu et al. presented a randomized algorithm for low-rank tensor train approximation of tensors based on randomized block Krylov subspace iteration \cite{YFCCQ23}.  It is worth mentioning that most of methods are based on the randomized SVD for matrices \cite{HMT11}, where the random Gaussian matrices are used. For large-scale tensors, this kind of methods is bottlenecked by the operation called the {\it tensor-times-matrix-chains}. To alleviate the computation cost, many works \cite{W14, ANW14, DSSW18} has led to the technique of TensorSketch which is ideally suited for sketching Kronecker products. In this way, the random matrix is very sparse and the accuracy could be also guaranteed with high probability. Recently, the technique of TensorSketch has been used for computing low-rank approximations of CP decomposition \cite{WTS15}, Tucker decomposition \cite{MB18, MS21} and tensor ring decomposition \cite{MB21, YL22a}. The main idea of these randomized algorithms is using TensorSketch to sketch the subproblems of  alternating least squares (ALS). However, the classic ALS has some drawbacks. One deficiency of ALS method is the swamp effect where plenty of iterations make the decrease of objective function almost null. Besides, the solution of the sketched ALS subproblem might be not unique. It is necessary to add a regularization term to the ALS subproblem. Motivated by the work of \cite{NDK08, LKN13}, we apply TensorSketch to compute the low-rank TT decomposition based on the regularized alternating least squares. The regularization term is actually a proximal term that penalizes the difference between the solution and the current iterate. Our contributions and the notations used in this paper are listed in Subsections \ref{sec_1-1} and \ref{sec_1-2}, respectively.

\subsection{Our Contributions}
\label{sec_1-1}
In this paper, we propose a new randomized proximal ALS algorithm for low-rank TT decomposition by using TensorSketch. Based on the regularized ALS, we incorporate TensorSketch to approximate the solution of large-scale subproblems rapidly, while the accuracy could be also guaranteed with sufficient sketch size. In summary, this paper makes the following contributions:
\begin{itemize}
    \item Based on the regularized ALS, a novel randomized algorithm is proposed for low-rank TT decomposition by using TensorSketch.
    \item The algorithm allows for efficient implementation via fast Fourier transform. The theoretical lower bounds of sketch size are estimated for approximating the optimal value of subproblems.
    \item Numerical experiments on synthetic and real-world data also demonstrate the effectiveness and efficiency of the proposed algorithm.
\end{itemize}

\subsection{Notations}
\label{sec_1-2}
Throughout this paper, scalars are denoted by lower case letters, e.g. $x$; vectors are denoted by bold lower case letters, e.g. $\bf x$; matrices are denoted by capital letters, e.g. $X$; tensors of order 3 or higher are denoted by calligraphic letters, e.g. $\mathcal{X}$.
For any positive integer $n$, denote $[n] = \{1, 2, \ldots, n \}$. For any matrix $A \in \mathbb{R}^{m \times n}$, the $i$th row vector and $j$th column vector of $A$ are denoted by $A(i,:)$ and $A(:,j)$, respectively. The Kronecker product of two matrices is denoted with “$\otimes$”. The identity matrix of size $n \times n$ is denoted by $I_n$. For any 3rd-order tensor $\mathcal{A} \in \mathbb{R}^{n_1 \times n_2 \times n_3}$, the $i$th slice of $\mathcal{A}$ is denoted by $\mathcal{A}(i) \in \mathbb{R}^{n_1 \times n_3}$, the left unfolding $\mathcal{A}^{\text{L}} \in \mathbb{R}^{n_1 n_2 \times n_3}$ is defined as
$ \mathcal{A}^{\text{L}}(i_1 + (i_2-1)n_1, i_3) = \mathcal{A}(i_1, i_2, i_3)$
and the right unfolding $\mathcal{A}^{\text{R}} \in \mathbb{R}^{n_1 \times n_2 n_3}$ is defined as
$\mathcal{A}^{\text{R}}(i_1,  i_2 + (i_3-1)n_2) = \mathcal{A}(i_1, i_2, i_3).$
For any tensor $ \mathcal{A} \in \mathbb{R}^{n_{1} \times n_{2} \times \cdots \times n_{d} }$, the mode-$k$ matricization $\mathcal{A}_{(k)}\in \mathbb{R}^{n_{k} \times \prod_{i\ne k} {n_{i}} }$ is defined as 
$\mathcal{A}_{(k)}(i_k,  j) = \mathcal{A}(i_1, i_2, \ldots, i_d),$
where 
$ j = 1+ \sum_{\substack{ s=1 \\ s\ne k}}^d(i_s-1) \prod_{\substack{t=1  \\  t\ne s}}^{s-1} n_t .$
The Frobenius norm of $ \mathcal{A} \in \mathbb{R}^{n_{1} \times n_{2} \times \cdots \times n_{d} }$ is defined as $\| {\cal A}\|_F = \sqrt{\sum_{i_1=1}^{n_1} \sum_{i_2=1}^{n_2} \cdots \sum_{i_d=1}^{n_d} \mathcal{A}(i_1, i_2, \ldots, i_d)^2 }.$

\begin{definition}[$k$-mode product \cite{KB09}]
The $k$-mode product of a tensor $\mathcal{X} \in \mathbb{R}^{n_1 \times n_2 \times \cdots \times n_d}$ and a matrix $A \in \mathbb{R}^{m \times n_k}$ is denoted by $\mathcal{X} \times_k A$ and is of size \\ $ {n_1 \times \cdots \times n_{k-1} \times m \times n_{k+1} \times \cdots \times n_d}$ with each element given by
$$(\mathcal{X} \times_k A)(i_1, \ldots, i_{k-1}, j ,i_{k+1}, \ldots, i_d) = \sum_{i_k=1}^{n_k} \mathcal{X} (i_1, \ldots, i_{k-1}, i_k ,i_{k+1}, \ldots, i_d) A(j, i_k) .$$
\end{definition}

\begin{definition}[Face-splitting product \cite{S97}]
Given $A \in \mathbb{R}^{m \times n_1}$ and $B \in \mathbb{R}^{m \times n_2}$, the face-splitting product $C = A \square B \in \mathbb{R}^{m \times n_1 n_2}$ is defined by the row-wise Kronecker product of matrices $A$ and $B$, i.e.,
$ C(i,:) = A(i,:) \otimes B(i,:)$ for any $i \in [m].$
\end{definition}

\begin{definition}[Slice-wise product]
Given $\mathcal{A} \in \mathbb{R}^{r_1 \times n \times r_2}$ and $\mathcal{B} \in \mathbb{R}^{r_2 \times n \times r_3}$, the slice-wise product $\mathcal{C} = \mathcal{A} \star \mathcal{B} \in \mathbb{R}^{r_1 \times n \times r_3}$ is defined by the slice-wise product of tensors $\mathcal{A}$ and $\mathcal{B}$, i.e.,
$ \mathcal{C}(i) = \mathcal{A}(i)  \mathcal{B}(i)$ for any $i \in  [n].$
\end{definition}

The above notations are summarized in \Cref{tab1}. 
The rest of this paper is organized as follows. We review some backgrounds on tensor train decomposition and TensorSketch in \Cref{sec_2}. In \Cref{sec_3}, we propose the randomized proximal ALS algorithm for low-rank TT decomposition and derive the fast computation of TensorSketch for proximal TT-ALS. The accuracy of TensorSketch for proximal TT-ALS is established in \Cref{sec_4}. In \Cref{sec_5}, numerical experiments for synthetic and real-world problems are presented to show the validity of proposed algorithm. Finally, the conclusions are drawn in \Cref{sec_6}.
\begin{table}[htbp]
\footnotesize
\caption{Description of notations.} \label{tab1}
\centering
\begin{tabular}{cc}
\hline \hline
\bf Notation & \bf Meaning \\ \hline
$x$ & Scalar \\ \hline
$\bf x$ & Vector \\ \hline
$X$ & Matrix \\ \hline
$\mathcal{X}$ & $d$th-order tensor $(d \ge 3)$ \\ \hline
$\left[ n \right]$ & The set $\{ 1, 2, \ldots, n\}$ \\ \hline
$A \left(i,:\right)$ & The $i$th row vector of matrix $A$ \\ \hline
$A \left(:,j\right)$ & The $j$th column vector of matrix $A$ \\ \hline
$\otimes$ & Kronecker product \\ \hline
$I_{n}$ & Identity matrix of size $n \times n$ \\ \hline
$\mathcal{A}(i)$ & The $i$th slice of 3rd-order tensor $\mathcal{A}$  \\ \hline
$\mathcal{A}^{\text{L}}$ & The left unfolding of 3rd-order tensor $\mathcal{A}$  \\ \hline
$\mathcal{A}^{\text{R}}$ & The right unfolding of 3rd-order tensor $\mathcal{A}$   \\ \hline
$\mathcal{A}_{(k)}$ & The mode-$k$ matricization of tensor $\mathcal{A}$  \\ \hline
$\left\| \cdot \right\|_{F}$ & Frobenius norm \\ \hline
$\times_{k}$ & $k-$mode product \\ \hline
$\square$ & Face-splitting product \\ \hline
$\star$ & Slice-wise product \\ \hline
\end{tabular}
\end{table}

\section{Backgrounds}
\label{sec_2}
\subsection{Tensor Train Decomposition}
A real $d$th-order tensor is a multidimensional array ${\cal A} \in \mathbb{R}^{n_1 \times n_2 \times \cdots \times n_d}$ that can be regarded as an extension of a matrix to its general $d$th order.
The challenge is that the number of tensor elements grows exponentially in $d$.
Even if each dimension (i.e. the number of possible values of each index) of a tensor is small, the storage cost for all elements is prohibitive for large $d$. The tensor train decomposition \cite{O11} gives an efficient way (in storage and computation) to alleviate so-called {\it curse of dimensionality}.

The main idea of TT decomposition is to re-express each element of a tensor ${\cal A} \in \mathbb{R}^{n_1 \times n_2 \times \cdots \times n_d} $ as
\begin{equation*}\label{e2.1}
{\cal A}(i_1, i_2, \cdots, i_d) = {\cal G}_1(i_1) {\cal G}_2(i_2) \cdots {\cal G}_d(i_d) ,
\end{equation*}
where ${\cal G}_k \in \mathbb{R}^{r_{k-1} \times n_k \times r_k}, k=1,2,\ldots, d$ are called TT-cores.
To make the matrix-by-matrix product a scalar, we set $r_0 = r_d = 1$.
The quantities $r_k$ are called TT-ranks. In fact, each core ${\cal G}_k$ is a third-order tensor with dimensions $r_{k-1}$, $n_k$ and $r_k$.
The tensor $\cal A$ is also denoted by
$ {\cal A} = [\![ {\cal G}_1,  {\cal G}_2, \ldots, {\cal G}_d ]\!]. $
Let $n = \max \{ n_1, n_2, \ldots, n_d \}$. It turns out that if all TT ranks are bounded by $r$, the storage of tensor train is $O(d n r^2 )$, which does not grow exponentially with $d$. The numerical stability of TT decomposition comes from the process of left and right orthogonalization \cite{O11}. Figure \ref{fig_1} illustrates the TT decomposition of a third-order tensor ${\cal A} \in \mathbb{R}^{n_1 \times n_2 \times n_3}$.

\begin{figure}[htbp]
	\centering
	\includegraphics[width=0.9\linewidth]{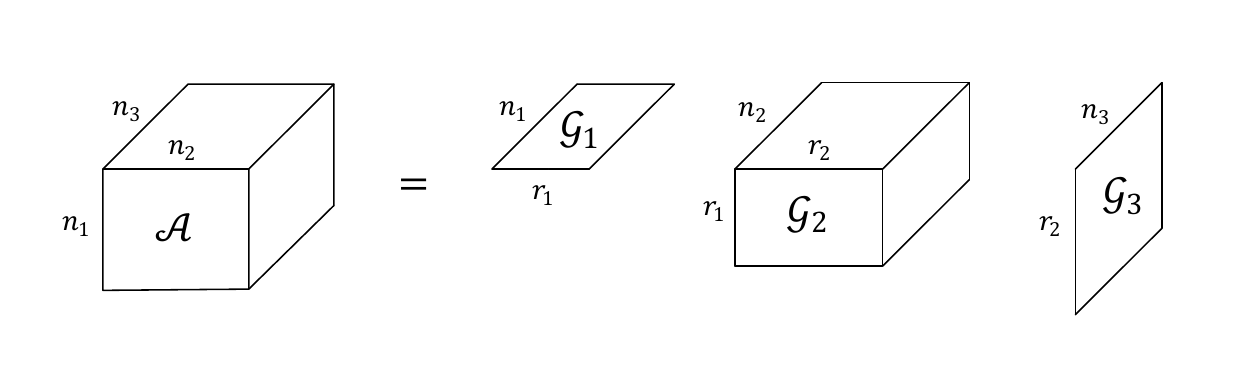}
	\caption{TT decomposition for ${\cal A} \in \mathbb{R}^{n_1 \times n_2 \times n_3}$}
	\label{fig_1}
\end{figure}

\subsection{TensorSketch}
TensorSketch is a variant of CountSketch that is designed specifically for tensors. It restricts the hash map to a specific format, enabling fast multiplication of the sketching matrix with the chain of Kronecker products. The hash map in TensorSketch maps the indices of the tensor to a format that allows for efficient multiplication with the sketching matrix. This enables the algorithm to compute the sketch of a tensor quickly for various tasks, such as tensor decomposition and regression. The use of TensorSketch can reduce the computational cost and memory requirements of tensor-based algorithms significantly. 
Before introducing the definition of TensorSketch, we first give the definitions of CountSketch and $k$-wise independent has map. For more details, the readers are referred to \cite{KN14, P13}.

\begin{definition}[CountSketch]
The CountSketch matrix is definied as $S=\Omega D\in \mathbb{R}^{m\times n} $, where
\begin{itemize}
    \item[(1)] $h: [n] \to [m]$ is a hash map such that $Pr[h(i) = j] =\frac{1}{m}$ for all $i\in [n]$ and  $j \in [m]$.
    \item[(2)] $\Omega \in \mathbb{R}^{m\times n} $ is a matrix with $\Omega(j,i)=1$ if $j = h(i)$ and $\Omega(j, i) = 0$ otherwise.
    \item[(3)] $D \in \mathbb{R}^{n\times n}$ is a diagonal matrix with diagonal a Rademacher vector $v \in \mathbb{R}^n$ (each entry is $+1$ or $-1$ with equal probability).
\end{itemize}
\end{definition}

\begin{definition}[$k$-wise independent]
A hash map $h: [n] \to [m]$ is called $k$-wise independent if the hash code of any fixed $i \in [n]$ is uniformly distributed in $[m]$, and the hash codes $h(i_1), h(i_2), \ldots, h(i_k)$ are independent random variables for any distinct $i_1, i_2, \ldots, i_k \in [n]$.
\end{definition}

Note that there is a bijection between the set of indices $i \in \left[ \prod_{k=1}^d n_k\right]$ and the $d$-tuples $(i_1, i_2, \ldots, i_d) \in  [n_{1}]\times [n_2] \times \cdots \times [n_{d}]$ according to the lexicographic order.
For simplicity, we use the notation $f(i) = f(i_{1},...,i_{d})$ for the funtion $f$ on the domain $[n_{1}]\times [n_2] \times \cdots \times [n_{d}]$, where $i =  1+ \sum_{s=1}^d(i_s-1) \prod_{t=1}^{s-1} n_t $.

\begin{definition}[TensorSketch]
The order $d$ TensorSketch matrix is defined as $S=\Omega D\in \mathbb{R}^{m\times \prod_{k=1}^{d} n_{k}}$, where
\begin{itemize}
    \item[$\bullet$] $h: [n_{1}]\times [n_2] \times \cdots \times [n_{d}]\to [m]$ is the hash map
    \begin{equation} \label{eqn_h}
    h(i_{1}, i_2, \ldots,i_{d}) =  \left( \sum_{k=1}^{d} \left( h_{k}(i_{k})-1 \right) \quad \text{mod} \ m \right)+1, \end{equation}
    where $h_k: [n_k] \to [m]$ is a 3-wise independent hash map for $k= 1, 2, \ldots, d$.
    \item[$\bullet$] $\Omega \in \mathbb{R}^{m\times \prod_{k=1}^{d} n_{k}} $ is a matrix with $\Omega(j,i)=1$ if $j = h(i)$ and $\Omega(j, i) = 0$ otherwise.
    \item[$\bullet$] $D \in \mathbb{R}^{\prod_{k=1}^{d} n_{k} \times \prod_{k=1}^{d} n_{k}}$ is a diagonal matrix with diagonal vector $v \in \mathbb{R}^{\prod_{k=1}^{d} n_{k}}$ given by
    \begin{equation} \label{eqn_v}
    v(i_1, i_2, \ldots, i_d) =  \prod_{k=1}^{d}v_{k}(i_{k}), \end{equation}
    where $v_k: [n_k] \to \{-1, 1\} $ is a 4-wise independent hash map for $k= 1, 2, \ldots, d$.
\end{itemize}
\end{definition}

To show the construction of TensorSketch matrix, an example is presented for the case $m=2, n_1=2, n_2=2$. Assume that the hash maps $h_1$ and $h_2$ are given by $h_1(1)=1$, $h_1(2)=2$ and $h_2(1)=1$, $h_2(2)=2$. By definition, the hash map $h$ is given by $h(1)=h(1,1)=1$, $h(2)=h(2,1)=2$, $h(3)=h(1,2)=2$, $h(4)=h(2,2)=1$. It follows that the corresponding matrix $ \Omega = 
       \begin{bmatrix}
         1 & 0 & 0 & 1 \\
         0 & 1 & 1 & 0
       \end{bmatrix}. $
If the Rademacher vectors $v_1$ and $v_2$ are given by $v_1 = (1, -1)^\top$ and $v_2=(-1, 1)^\top$, the diagonal vector $v$ is defined as $v=(-1, 1, 1, -1)^\top$. As a result, the TensorSketch matrix $S \in \mathbb{R}^{2 \times 4}$ is constructed as
$ S = \begin{bmatrix}
        -1 &  0 &  0 & -1 \\
         0 &  1 &  1 &  0
       \end{bmatrix}.$

It is well-known that if $h_k$ is 3-wise independent for $k \in [d]$, the hash map $h$ constructed in TensorSketch is also 3-wise independent \cite{CW79, PT12}. It is worth mentionting that the TensorSketch matrix is very sparse and does not need to be constructed in the full form. As we can see later, the sketch of a tensor with TT format could be implemented efficiently via fast Fourier transform.

\section{Fast computation of TensorSketch for proximal TT-ALS}
\label{sec_3}
Given a tensor ${\cal A} \in \mathbb{R}^{n_1 \times n_2 \times \cdots \times n_d}$ and TT-ranks $\{ r_k\}_{k=0}^d$, the goal of low-rank tensor train decomposition is to minimize the objective function
\begin{equation}\label{lrtt}
f(\mathcal{G}_{1}, \mathcal{G}_{2}, \ldots, \mathcal{G}_{d})=\frac{1}{2}\left\| [\![\mathcal{G}_{1}, \mathcal{G}_{2}, \ldots, \mathcal{G}_{d}]\!] - \mathcal{A} \right\|_{F}^{2},
\end{equation}
where $\mathcal{G}_{k} \in \mathbb{R}^{r_{k-1} \times n_k \times r_{k} }$ for $k \in [d]$ and $r_{0}=r_{d}=1$. To minimize (\ref{lrtt}), TT-ALS is the most widely used algorithm which updates each TT-core alternatively while the other TT-cores are fixed \cite{HRS12}. To be specific, for $k \in [d]$, the TT-core $\mathcal{G}_{k}$ is updated by solving the corresponding least squares problem, i.e., 
\begin{equation}\label{lrtt_sub} 
      \begin{split}
        \min_{{\cal G}_k} \quad &  \frac{1}{2} \left\|  \left( G_{>k}^\top  \otimes G_{<k} \right) {\cal G}_{k(2)}^\top - {\cal A}_{(k)}^\top \right\|_F^2  
      \end{split}
\end{equation} 
where $G_{<k} = \texttt{reshape}\left( [\![ {\cal G}_1, {\cal G}_2, \ldots, {\cal G}_{k-1}  ]\!], \prod_{i<k} n_i, r_{k-1}\right)$ and \\ $G_{>k} = \texttt{reshape}\left( [\![ {\cal G}_{k+1}, {\cal G}_{k+2}, \ldots, {\cal G}_{d}  ]\!], r_k, \prod_{i>k} n_i  \right).$ Here we define $G_{<1} = G_{>d} = 1$.
Under mild assumptions, the local convergence of TT-ALS is guaranteed \cite{RU13}. 
However, the solution of (\ref{lrtt_sub}) may be not unique. Here we consider TT-ALS with proximal regularization:
    \begin{equation}\label{lrtt_sub_prox} 
      \begin{split}
        \min_{{\cal G}_k} \quad &  \frac{1}{2} \left\|  \left( G_{>k}^\top  \otimes G_{<k} \right) {\cal G}_{k(2)}^\top - {\cal A}_{(k)}^\top \right\|_F^2  + \frac{\sigma}{2} \left\| {\cal G}_k - {\cal G}_k^{(t)} \right\|_F^2 .
      \end{split}
    \end{equation} 

However, the cost of solving subproblem (\ref{lrtt_sub_prox}) is $O(n^d r^2)$ which is impractical for large-scale problems, where $n=\max\{n_1, n_2, \ldots, n_d\}$ and $r=\max\{ r_0, r_1, \ldots, r_d \}$. The idea is to find a sketching matrix $S \in \mathbb{R}^{m \times \prod_{i \ne k}n_i}$ to solve the sketched proximal least squares problem:
    \begin{equation}\label{lrtt_sub4} 
      \begin{split}
        \min_{{\cal G}_k} \quad &  \frac{1}{2} \left\|  S H_k {\cal G}_{k(2)}^\top - S {\cal A}_{(k)}^\top \right\|_F^2  + \frac{\sigma}{2} \left\| {\cal G}_k - {\cal G}_k^{(t)} \right\|_F^2 
      \end{split},
    \end{equation} 
    where $H_k = G_{>k}^\top  \otimes G_{<k}$ and ${\cal G}_k^{(t)}$ denotes ${\cal G}_k$ at the $t$th iteration. It follows that 
    \begin{equation}\label{ite} {\cal G}_{k(2)}^{(t+1)} = \left( {\cal A}_{(k)} S S^\top H_k + \sigma {\cal G}_{k(2)}^{(t)}  \right) \left( H_k^\top S^\top S H_k + \sigma I \right)^{-1}  .\end{equation}
We can see that if $\sigma >0$, (\ref{ite}) is always well-defined even though $H_k^\top S^\top S H_k$ is singular. To make the computation practical, the choice of the sktching matrix $S$ should satisfy two requirements. First, the computation of $S \left( G_{>k}^\top  \otimes G_{<k} \right) $ and $S {\cal A}_{(k)}^\top$ could be implemented efficiently. Second, the solution of (\ref{lrtt_sub4}) should be near the exact solution of (\ref{lrtt_sub_prox}).

Here we use TensorSketch to constructure the sketching matrix $S$ in (\ref{lrtt_sub4}). TensorSketch is a special type of CountSketch, where the hash map is restricted to a special format to allow fast multiplication of the sketching matrix with the chain of Kronecker products. In this section, we show the fast computation of TensorSketch for proximal TT-ALS. The accuracy of TensorSketch for proximal TT-ALS will be shown in the next section. 
Denote by $F \in \mathbb{R}^{m \times m}$ the Fourier transform matrix, i.e., $$ F = 
       \begin{bmatrix}
         1     & 1       & \cdots & 1          \\
         1     & w     & \cdots & w^{m-1}  \\
        \vdots & \vdots  & \ddots & \vdots     \\
        1      & w^{m-1} & \cdots & w^{2(m-1)}
       \end{bmatrix} $$
where $w = e^{-{\bf i} \cdot 2\pi  /m}$. 

\begin{lemma}\label{lem1}
    Let $S \in \mathbb{R}^{m \times n_1 n_2}$ be the TensorSketch matrix generated by CountSketch matrices $S_1 \in \mathbb{R}^{m \times n_1 }$ and $S_2 \in \mathbb{R}^{m \times  n_2}$. It holds that $$ S = F^{-1} [(FS_2) \square (FS_1)],$$
    where $F \in \mathbb{R}^{m \times m}$ is the Fourier transform matrix.
\end{lemma} 
\begin{proof} For $i=1,2$, let $h_i$ and $v_i$ be the hash map and Rademacher vector of CountSketch matrix $S_i$, respectively. The hash maps $h$ and $v$ of TensorSketch matrix $S$ are constructed by (\ref{eqn_h}) and (\ref{eqn_v}) respectively. For $s \in [m]$ and $i \in [n_1 n_2]$, by definition, the $(s, i)$th element of $FS$ could be written as $v(i) w_s^{h(i)-1}$, where $w_s = e^{-{\bf i} \cdot 2\pi (s-1) /m }$. Since $w_s^m = 1$, it follows that
\begin{subequations}
\begin{align*}
 (FS)(s, i) & =  v(i) w_s^{h(i)-1} = v(i_1, i_2) w_s^{h_1(i_1) + h_2(i_2) -2} \\ 
 & = v_1(i_1) w_s^{h_1(i_1)-1} v_2(i_2) w_s^{h_2(i_2) -1} = (FS_1)(s, i_1) (FS_2)(s, i_2), 
\end{align*}
\end{subequations}
where $(i_1, i_2) \in [n_1] \times [n_2]$ is the pair corresponding to $i \in [n_1 n_2]$ according to the lexicographic order. This means $(FS)(s, :) = (FS_2)(s, :) \otimes (FS_1)(s, :)$ for $ s \in [m]$. So we have $FS = (FS_2) \square (FS_1)$, which completes the proof.
\end{proof}

Similarly, we have the following corollaries. Here we use the facts that the face-splitting product satisfies the associative law and $(D \square C)(B \otimes A) = (DB) \square (CA)$ for $A\in \mathbb{R}^{n_1 \times r_1}$, $B\in \mathbb{R}^{n_2 \times r_2}$, $C \in \mathbb{R}^{m \times n_1}$ and $D \in \mathbb{R}^{m \times n_2}$ \cite{S99}.
\begin{corollary}\label{cor1}
    Let $S \in \mathbb{R}^{m \times n_1 n_2 \cdots n_d}$ be the TensorSketch matrix generated by CountSketch matrices CountSketch matrices $S_k \in \mathbb{R}^{m \times n_k }$, $k \in [d]$. It holds that $$ S = F^{-1} [(FS_d) \square  (FS_{d-1}) \square \cdots \square (FS_1)],$$
    where $F \in \mathbb{R}^{m \times m}$ is the Fourier transform matrix.
\end{corollary}

\begin{corollary}\label{cor2}
    Let $S \in \mathbb{R}^{m \times n_1 n_2}$ be the TensorSketch matrix generated by CountSketch matrices $S_1 \in \mathbb{R}^{m \times n_1 }$ and $S_2 \in \mathbb{R}^{m \times  n_2}$. It holds that $$ S(B \otimes A)=F^{-1}[(F S_2 B) \square (F S_{1} A)] $$ for any $A\in \mathbb{R}^{n_1 \times r_{1}}$ and $B \in \mathbb{R}^{ n_2 \times r_{2}}$, where $F \in \mathbb{R}^{m \times m}$ is the Fourier transform matrix.
\end{corollary}

\begin{theorem}\label{thm1}
    Let $S \in \mathbb{R}^{m \times \prod_{i \in [d] \setminus \{ k\}} n_i} $ be the TensorSketch matrix generated by CountSketch matrices $S_i \in \mathbb{R}^{m \times n_i }$, $i \in [d]\setminus \{k\}$. It holds that 
     $$ S(G_{> k}^{\top}\otimes G_{< k}) =F^{-1}[(F S_{> k}G_{>k}^{\top})\square (F S_{< k}G_{< k})], $$
     where $S_{<k} \in \mathbb{R}^{m \times n_1 \cdots n_{k-1}}$ is the TensorSketch matrix generated by CountSketch matrices $\{ S_i\}_{i<k}$ and $S_{>k}  \in \mathbb{R}^{m \times n_{k+1} \cdots n_d}$ is the TensorSketch matrix generated by CountSketch matrices $\{ S_i\}_{i>k}$.
\end{theorem}
\begin{proof} According to Corollary \ref{cor1} and Corollary \ref{cor2}, one could obtain that
\begin{eqnarray*}
    FS(G_{> k}^{\top}\otimes G_{< k}) &=& [(FS_d) \square \cdots \square (FS_{k+1}) \square (FS_{k-1}) \square \cdots \square (FS_1)] (G_{> k}^{\top}\otimes G_{< k}) \\
    &=& [(FS_{>k}) \square (FS_{<k})] (G_{> k}^{\top}\otimes G_{< k}) \\
    &=&  (F S_{> k}G_{>k}^{\top})\square (F S_{< k}G_{< k}). 
\end{eqnarray*}
By left multiplying the matrix $F^{-1}$ on both sides of the equation above, we get the desired conclusion.
\end{proof}

By making full use of the structure of TensorSketch matrix and TT decomposition, the computations of $F S_{< k}G_{< k}$ and $F S_{> k}G_{>k}^{\top}$ could be divided into the slice-wise products of corresponding TT-cores, respectively. It turns out that the computation of $S \left( G_{>k}^\top  \otimes G_{<k} \right)$ could be implemented efficiently via fast Fourier transform.

\begin{theorem}\label{thm2}
Let $S_{<k}$ and $S_{>k}$ be the TensorSketch matrices defined as Theorem \ref{thm1}. It holds that
$$F S_{<k}G_{<k}=\left[(\mathcal{G}_{1}\times_{2}F  S_1) \star (\mathcal{G}_{2}\times_{2}F S_2)\star \cdots \star (\mathcal{G}_{k-1}\times_{2}F S_{k-1}) \right]^{\text{L}}$$ and $$F S_{>k}G_{>k}^{\top}=\left[\left( (\mathcal{G}_{k+1}\times_{2} F S_{k+1}) \star (\mathcal{G}_{k+2}\times_{2}F S_{k+2}) \star \cdots \star (\mathcal{G}_{d}\times_{2}F S_d )\right)^{\text{R}}\right]^{\top}.$$
\end{theorem}
\begin{proof} We only prove the first equality since the second equality could be proved similarly.
For $i \in [d]\setminus \{k\}$, let $h_i$ and $v_i$ be the hash map and Rademacher vector of CountSketch matrix $S_i$, respectively. Since $S_{<k}$ is the TensorSketch matrix generated by $\{ S_i\}_{i<k}$, the hash maps $h_{<k}$ and $v_{<k}$ of $S_{<k}$ are constructed from $\{ h_i\}_{i<k}$ and $\{ v_i\}_{i<k}$ according to (\ref{eqn_h}) and (\ref{eqn_v}) respectively. By definition, one could obtain that for $s \in [m]$, the $s$th row of $F S_{<k}G_{<k}$ could be written as 
\begin{eqnarray*}
    (F S_{<k}G_{<k})(s, :) &=& \sum_{i=1}^{n_1 \cdots n_{k-1}} v_{<k}(i) w_s^{h_{<k}(i)-1} G_{<k} (i, :) \\
    &=& \sum_{i_1=1}^{n_1} \cdots \sum_{i_{k-1}=1}^{n_{k-1}} v_1(i_1) \cdots v_{k-1}(i_{k-1}) w_s^{h_1(i_1) + \cdots + h_{k-1}(i_{k-1})-k+1} {\cal G}_1(i_1) \cdots {\cal G}_{k-1}(i_{k-1})  \\
    &=& \sum_{i_1=1}^{n_1} v_1(i_1) w_s^{h_1(i_1)-1} {\cal G}_1(i_1) \cdots \sum_{i_k=1}^{n_{k-1}} v_{k-1}(i_{k-1}) w_s^{h_{k-1}(i_{k-1})-1} {\cal G}_{k-1}(i_{k-1}) \\
    &=& (\mathcal{G}_{1}\times_{2}F  S_1)(s) \cdots (\mathcal{G}_{k-1}\times_{2}F S_{k-1})(s)
\end{eqnarray*}
where $w_s = e^{-{\bf i} \cdot 2\pi (s-1) /m }$ and $(i_1, \ldots, i_{k-1}) \in [n_1] \times \cdots \times [n_{k-1}]$ is the tuple corresponding to $i \in \left[ \prod_{i<k} n_i \right]$ according to the lexicographic order. So the first equality holds.
\end{proof}

\begin{theorem}\label{thm3}
    Let $n=\max\{n_1, n_2, \ldots, n_d\}$ and $r=\max\{ r_0, r_1, \ldots, r_d \}$. The products of $F S_{> k}G_{>k}^{\top} $ and $F S_{< k}G_{< k}$ could be computed at a cost of $O\left((m+m\log m+n)(d-1)r^2 \right)$. As a result, the computation cost of $S \left( G_{>k}^\top  \otimes G_{<k} \right) $ is $O\left( (m+m\log m) d r^2 +n(d-1)r^2 \right)$.
\end{theorem}
\begin{proof} For $i\in [d]\setminus \{ k\}$, the computation cost of ${\cal G}_i \times_2 S_i$ is $O(nr^2)$ due to the structure of CountSketch matrices. It follows that the computation cost of ${\cal G}_i \times_2 F S_i$ is $O(nr^2+ r^2 m\log m )$. By recursion, the slice-wise products of $(\mathcal{G}_{1}\times_{2}F  S_1) \star \cdots \star (\mathcal{G}_{k-1}\times_{2}F S_{k-1})$ and $(\mathcal{G}_{k+1}\times_{2} F S_{k+1}) \star \cdots \star (\mathcal{G}_{d}\times_{2}F S_d )$ could be computed at the cost of $O\left(m(k-1)r^2\right)$ and $O\left(m(d-k)r^2 \right)$, respectively. By Theorem \ref{thm2}, the total computation cost of $F S_{> k}G_{>k}^{\top} $ and $F S_{< k}G_{< k}$ is $O\left((m+m\log m+n)(d-1)r^2 \right)$. Furthermore, the cost of face-splitting product $(F S_{> k}G_{>k}^{\top})\square (F S_{< k}G_{< k})$ is $O(mr^2)$. By Theorem \ref{thm1}, the total computation cost of $S \left( G_{>k}^\top  \otimes G_{<k} \right)$ is $O\left( (m+m\log m) d r^2 +n(d-1)r^2 \right)$ when adding the cost of inverse fast Fourier transform.
\end{proof}

The proposed algorithm for low-rank tensor train decomposition using TensorSketch is described in Algorithm \ref{alg_TT-TS}. It is worth mentioning that the computation cost of $S \left( G_{>k}^\top  \otimes G_{<k} \right)$ goes linearly with the order $d$, whereas the naive matrix multiplication would cost $O(mn^{d-1}r^2)$ which goes exponentially with the order $d$. The special structure of TensorSketch matrices makes the computation more practical for large-scale problems.
In fact, the mode products $\{ {\cal G}_i \times_2 F S_i \}_{i \in [d]\setminus \{k\}}$ could be also computed in parallel to reduce the cost. Moreover, there is no need to store the whole tensor $\cal A$ since only a few fibers are used to compute $S {\cal A}_{(k)}^\top$ in (\ref{lrtt_sub4}). In particular, if $\cal A$ is sparse, the computation cost of $S {\cal A}_{(k)}^\top$ is $O\left(nnz({\cal A})\right)$, where $nnz({\cal A})$ denotes the number of nonzero elements of $\cal A$.

\begin{algorithm}
\caption{Low-rank Tensor Train Decomposition using TensorSketch (TT-TS)}
\label{alg_TT-TS}
\begin{algorithmic}
\REQUIRE{ $\mathcal{A}\in \mathbb{R}^{n_{1}\times n_2 \cdots \times n_{d}}$, TT-ranks $\{ r_k \}_{k=1}^{d-1}$, sketch size $m$ and $\sigma > 0$ }
\ENSURE{ TT-cores $\{{\cal G}_k\}_{k=1}^d$ }
\STATE{Initialize TT-cores $\{{\cal G}_k\}_{k=1}^d$ of prescribed ranks }
\WHILE{termination condition is not satisfied }
\STATE{Execute right-to-left orthogonalizion for $\{{\cal G}_k\}_{k=1}^d$ }
\STATE{Define TensorSketch operators $S_k\in \mathbb{R}^{m\times n_k }$, $k\in [d]$ }
\FOR{$k = 1,2, \ldots, d $ }
\STATE{Compute $S(G_{> k}^{\top}\otimes G_{< k}) $ by $F^{-1}[(F S_{> k}G_{>k}^{\top})\square (F S_{< k}G_{< k})]$ }
\STATE{Compute $S {\cal A}_{(k)}^\top$ by using some fibers of $\cal A$ }
\STATE{Update ${\cal G}_k$ according to (\ref{ite})  }
\IF{$k<d$}
\STATE{$[Q, R] \leftarrow$ compute QR decomposition of $G_k^\text{L}$ }
\STATE{${\cal G}_k \leftarrow \texttt{reshape}(Q, r_{k-1}, n_k, r_k)$ }
\STATE{${\cal G}_{k+1} \leftarrow \texttt{reshape}(R G_{k+1}^\text{R} , r_k, n_{k+1}, r_{k+1})$ } 
\ENDIF
\ENDFOR
\ENDWHILE
\RETURN{$\mathcal{G}_{1}, {\cal G}_2, \ldots,\mathcal{G}_{d}$} 
\end{algorithmic}
\end{algorithm}

\section{Accuracy of TensorSketch for proximal TT-ALS}
\label{sec_4}
In this section, we start from the fact that TensorSketch is an oblivious subspace embedding to derive theoretical results of sketch size for approximating the optimal value of (\ref{lrtt_sub_prox}).  We first introduce that the approximate matrix product property of TensorSketch matrices. 

\begin{lemma}[Approximate matrix product \cite{ANW14}]\label{lem2} Let $S \in \mathbb{R}^{m \times n_1n_2 \cdots n_q}$ be a TensorSketch matrix generated by 3-wise independent hash maps $h_i: [n_i] \rightarrow [m]$ and 4-wise inpependent sign functions $v_i: [n_i] \rightarrow \{1, -1\}$, where $i=1,2,\ldots,q$.  Let $A$ and $B$ be matrices with $n_{1}n_{2}\cdots n_{q}$ rows. For $m\ge (2+3^{q})/(\epsilon_0^{2}\delta_0),$ we have
    $$Pr\left [ \left \| A^{\top}S^{\top}SB-A^{\top}B \right \|^{2}_{F}\le \epsilon_0^{2}\left \| A \right \|^{2}_{F}\left \| B \right \|_{F}^{2} \right ] \ge 1-\delta_0 ,$$
where $Pr[\cdot]$ denotes the probability of a random event.
\end{lemma}


\begin{lemma}\label{lem3}
Let $S\in \mathbb{R}^{m\times n_{1}n_{2}\cdots n_{q}}$ be the TensorSketch matrix defined as in Lemma \ref{lem2}. Let A and B be matrices with $n_{1}n_{2}\cdots n_{q}$ rows such that $\|A\|_F^2 \leq s$. If $ m= \max{\left\{ 8s^{2}(2+3^{q})/\delta, \ 8s(2+3^{q})/(\epsilon \delta)\right\}},$ the inequalities
$    \left \| A^{\top}S^{\top}SB-A^{\top}B  \right \|_{F}^{2}\le \frac{\epsilon}{4} \left \| B \right \|^{2}_{F}$
and 
$    \left \| A^{\top}S^{\top}SA-A^{\top}A  \right \|_{F}^{2}\le \frac{1}{4} $
hold simultaneously with probability at least $1-\delta$.
\end{lemma}
\begin{proof} According to Lemma \ref{lem2}, for $m \geq 8s(2+3^q)/(\epsilon \delta)$, the inequality
\begin{equation}\label{eqn_AB}
    \left \| A^{\top}S^{\top}SB-A^{\top}B  \right \|_{F}^{2}\le \frac{\epsilon}{4s} \|A \|_F^2 \left \| B \right \|^{2}_{F} \le \frac{\epsilon}{4} \left \| B \right \|^{2}_{F}
\end{equation}
holds with probability at least $1 - \delta/2$ by setting $\epsilon_0= \sqrt{\epsilon/ 4 s} $ and $\delta_0 = \delta/2$. Again, for $m \geq 8 s^2 (2+3^q)/\delta$, the inequality
\begin{equation}\label{eqn_AA}
    \left \| A^{\top}S^{\top}SA-A^{\top}A  \right \|_{F}^{2}\le \frac{1}{4s^2} \|A \|_F^4  \le \frac{1}{4} 
\end{equation}
holds with probability at least $1 - \delta/2$ by setting $\epsilon_0= \sqrt{1/4s^2} $ and $\delta_0 = \delta/2$. If $m= \max{\left\{ 8s^{2}(2+3^{q})/\delta, \ 8s(2+3^{q})/(\epsilon \delta)\right\}}$, we have
\begin{subequations}
\begin{align*}
      Pr\left[ \text{(\ref{eqn_AB}) and  (\ref{eqn_AA}) hold} \right]&= 1- Pr\left[ \text{(\ref{eqn_AB}) does not hold or (\ref{eqn_AA}) does not hold} \right] \\ 
                                    &\geq 1- Pr\left[\text{(\ref{eqn_AB}) does not hold}\right] - Pr\left[\text{(\ref{eqn_AA}) does not hold}\right]\\
                                    &\geq 1- \frac{\delta}{2} - \frac{\delta}{2} = 1 - \delta. 
\end{align*}
\end{subequations}
\end{proof}

\begin{theorem}[TensorSketch for Least Squares]\label{thm4} Given a full-rank matrix $P\in \mathbb{R}^{n_{1}n_{2}\cdots n_{q}\times s}$ with $n_{1}n_{2}\cdots n_{q}>s$, and $B\in \mathbb{R}^{n_{1}n_{2}\cdots n_{q}\times n}$, let $S\in \mathbb{R}^{m\times n_{1}n_{2}\cdots n_{q}}$ be the TensorSketch matrix defined as in Lemma \ref{lem2}. Denote $X_{opt} = \arg \min_{X}\left \| PX-B \right \| _{F}$, $\tilde{X}_{opt}  = \arg \min_{X}\left \| SPX-SB \right \| _{F}$ and $B^{\perp} = PX_{opt}-B$.
If $$m= \max{\left\{ 8s^{2}(2+3^{q})/\delta, \ 8s(2+3^{q})/(\epsilon \delta)\right\}},$$ the following approximation holds with probability at least $1-\delta$,
$$\left \| P\tilde{X}_{opt}-B \right \|_{F}^{2}\le \left(1+\epsilon \right) \left \| B^{\perp} \right \|_{F}^{2}. 
$$
\end{theorem}
\begin{proof} Define the reduced QR decomposition of $P$, i.e., $P=Q_{P}R_{P}$ where $Q_{P} \in \mathbb{R}^{n_1 n_2 \cdots n_q \times s}$ satisfies $Q_P^\top Q_P = I_s$ and $R_{P} \in \mathbb{R}^{s \times s}$ is upper triangular. note that $R_P$ is nonsingular since $P$ is full-rank. The sketched least squares problem is rewritten as
\begin{subequations}
\begin{align*}
       \min_{X}\left \| SPX-SB \right \|_{F} &= \min_{X}\left \| SPX-S(PX_{opt}-B^{\perp }) \right \|_{F}\\
       &=\min _{X}\left \| SQ_{P}R_{P}(X-X_{opt})+SB^{\perp} \right \|_{F}.
\end{align*}
\end{subequations}
From the optimality condition, one could obtain that
\begin{equation}\label{eqn_opt}
(SQ_{P})^{\top}SQ_{P}R_{P}({X}_{opt}-\tilde{X}_{opt})=(SQ_{P})^{\top}SB^{\perp }.
\end{equation}
Similarly, we have $Q_{P}^{\top}(PX_{opt}-B) = Q_{P}^{\top}B^{\perp }=0$ since $X_{opt} = \arg \min_{X}\left \| PX-B \right \| _{F}$.

According to Lemma \ref{lem3}, since $\| Q_P\|_F^2 = s$, the inequalities
\begin{equation}\label{eqn_QB}
    \left \| Q_{P}^{\top}S^{\top}SB^{\perp} \right \|_{F}^{2}=\left \| Q_{P}^{\top}S^{\top}SB^{\perp}-Q_{P}^{\top}B^{\perp } \right \|_{F}^{2}\le \frac{\epsilon}{4} \left \| B^{\perp } \right \|^{2}_{F}
\end{equation}
and
\begin{equation}\label{eqn_QQ}
    \left \| Q_{P}^{\top}S^{\top}S Q_P - I_s \right \|_{F}^{2}=\left \| Q_{P}^{\top}S^{\top}S Q_P -Q_{P}^{\top}Q_P \right \|_{F}^{2}\le \frac{1}{4}
\end{equation}
hold simultaneously with probability at least $1-\delta$. From (\ref{eqn_QQ}), one could derive that 
\begin{equation}\label{eqn_sig}
    \sigma _{min}^{2}(SQ_{P})\ge \frac{1}{2} ,
\end{equation}
where $\sigma_{min}(SQ_{P})$ is the minmal singular value of $SQ_{P}$. Based on (\ref{eqn_opt}), (\ref{eqn_QB}) and (\ref{eqn_sig}), we obtain
\begin{subequations}
\begin{align*}
 \left \| P\tilde{X}_{opt}-PX_{opt}  \right \|_{F}^{2}&= \left \| R_{P}\tilde{X}_{opt}-R_{P}X_{opt}  \right \|_{F}^{2}\\ 
                                    &\le 4\left \|(SQ_{P})^{\top}SQ_{P}R_{P}({X}_{opt}-\tilde{X}_{opt}) \right \|_{F}^{2}\\
                                    &=4\left \| Q_{P}^{\top}S^{\top}SB^{\perp } \right \|_{F}^{2}\le \epsilon \left \| B^{\perp} \right \|_{F}^{2}.
\end{align*}
\end{subequations}
Thus, with probability at least $1 -\delta$, we have
\begin{subequations}
\begin{align*}
      \left \| P\tilde{X}_{opt}-B  \right \|_{F}^{2}&= \left \| P\tilde{X}_{opt}-PX_{opt}+PX_{opt}-B  \right \|_{F}^{2}\\ 
                                    &=\left \| P\tilde{X}_{opt}-PX_{opt} \right \|_{F}^{2}+\left \| B^{\perp} \right \|_{F}^{2}\\
                                    &\le \left(1+ \epsilon \right) \left \| B^{\perp} \right \|_{F}^{2}.  
\end{align*}
\end{subequations}
\end{proof}

\begin{theorem}[TensorSketch for Least Squares with Proximal]\label{thm5} Given a full-rank matrix $P\in \mathbb{R}^{n_{1}n_{2}\cdots n_{q}\times s}$ with $n_{1}n_{2}\cdots n_{q}>s$, $B\in \mathbb{R}^{n_{1}n_{2}\cdots n_{q}\times n}$, $C \in \mathbb{R}^{s \times n}$ and $\sigma >0$, let $S\in \mathbb{R}^{m\times n_{1}n_{2}\cdots n_{q}}$ be the TensorSketch matrix defined as in Lemma \ref{lem2}. Denote $X_{opt}=\arg \min_{X}\frac{1}{2}\left \| PX-B\right \|_{F}^{2}+\frac{\sigma}{2}\left\|X-C\right\|_{F}^{2}$ and $\tilde{X}_{opt}=\arg \min_{X}\frac{1}{2}\left \| SPX-SB\right \|_{F}^{2}+\frac{\sigma}{2}\left\|X-C\right\|_{F}^{2}$. If $m= \max{\left\{ 8s^{2}(2+3^{q})/\delta, \ 8s(2+3^{q})/(\epsilon \delta)\right\}}$, the approximation
    $$\frac{1}{2}\left \| P\tilde{X}_{opt}-B\right \|_{F}^{2}+\frac{\sigma}{2}\left\|\tilde{X}_{opt}-C\right\|_{F}^{2}\leq \left(1+\epsilon\right)\cdot OPT $$ 
holds with probability at least $1- \delta$, where $OPT=\frac{1}{2}\left \| P{X}_{opt}-B\right \|_{F}^{2}+\frac{\sigma}{2}\left\|{X}_{opt}-C\right\|_{F}^{2}$.
\end{theorem}
\begin{proof} Since $P$ is full-rank, the matrix $\begin{bmatrix} P\\ \sqrt{\sigma}I_s  \end{bmatrix} \in \mathbb{R}^{(n_{1}n_{2}\cdots n_{q}+s)\times s}$ is also full-rank. Let $\tilde{P} = \begin{bmatrix} Q_1\\ Q_2 \end{bmatrix} \in \mathbb{R}^{(n_{1}n_{2}\cdots n_{q}+s)\times s}$ be the othonormal basis for the column space of $\begin{bmatrix} P\\ \sqrt{\sigma}I_s  \end{bmatrix}$, where $Q_1 \in \mathbb{R}^{n_1 n_2 \cdots n_q \times s}$ and $Q_2 \in \mathbb{R}^{s \times s}$. It follows that for any $X\in \mathbb{R}^{s \times n}$, there is a unique $Y\in \mathbb{R}^{s \times n}$ such that $\tilde{P}Y=\begin{bmatrix} P\\ \sqrt{\sigma}I_s \end{bmatrix} X$, and vice versa. Let $\tilde{B} = \begin{bmatrix} B\\ \sqrt{\sigma}C \end{bmatrix}$. Since
$\frac{1}{2}\left \| \tilde{P}Y- \tilde{B} \right \|_{F}^{2}=\frac{1}{2} \left \| \begin{bmatrix} P\\ \sqrt{\sigma}I_s \end{bmatrix} X-\begin{bmatrix} B\\ \sqrt{\sigma}C \end{bmatrix}\right \|_{F}^{2}= \frac{1}{2}\left \| PX-B\right \|_{F}^{2}+\frac{\sigma}{2} \left\|X-C\right\|_{F}^{2},$
the two optimization problems are equivalent, i.e.,
$$\min_{X} \frac{1}{2}\left \| PX-B\right \|_{F}^{2}+\frac{\sigma}{2}\left\|X-C\right\|_{F}^{2} \Longleftrightarrow  \min_{Y} \frac{1}{2}\left \| \tilde{P}Y-\tilde{B}   \right \|_{F}^{2}.$$
Let $Y_{opt}=\arg\min_{Y} \frac{1}{2}\left \| \tilde{P}Y-\tilde{B}   \right \|_{F}^{2}$, so that $\tilde{P}Y_{opt}=\begin{bmatrix} P  \\ \sqrt{\sigma} I_s \end{bmatrix} X_{opt}$. Next, we define $\tilde{S}$ to be 
  $\begin{bmatrix}
    S&0 \\
    0&I_{s}
   \end{bmatrix}$.
Let $\tilde{Y}_{opt}=\arg\min_{Y} \frac{1}{2}\left \| \tilde{S}\tilde{P}Y- \tilde{S} \tilde{B}  \right \|_{F}^{2}$. Similarly, we have $\tilde{P} \tilde{Y}_{opt}=\begin{bmatrix} P  \\ \sqrt{\sigma} I_s \end{bmatrix} \tilde{X}_{opt}$ since
$\frac{1}{2}\left \| \tilde{S}\tilde{P}Y- \tilde{S}\tilde{B} \right \|_{F}^{2} = \frac{1}{2}\left \| SPX-SB\right \|_{F}^{2}+\frac{\sigma}{2} \left\|X-C\right\|_{F}^{2}.$

Let $\tilde{B}^{\perp}=\tilde{P}Y_{opt}-\tilde{B}$ and $B^{\perp}=Q_1 Y_{opt} -B$. According to Lemma \ref{lem3}, the inequalities
$
    \left \| Q_{1}^{\top}S^{\top}SB^{\perp}-Q_{1}^{\top}B^{\perp } \right \|_{F}^{2}\le \frac{\epsilon}{4} \left \| B^{\perp } \right \|^{2}_{F}
$
and
$
    \left \| Q_{1}^{\top}S^{\top}S Q_1 -Q_{1}^{\top}Q_1 \right \|_{F}^{2}\le \frac{1}{4}
$
hold simultaneously with probability at least $1-\delta$ since $\left\|Q_1 \right\|_F^2 \leq \left\| \tilde{P} \right\|_F^2 \leq s$. It follows that
$$ \left \|\tilde{P}^{\top}\tilde{S}^{\top}\tilde{S}\tilde{B}^{\perp}-\tilde{P}^{\top}\tilde{B}^{\perp} \right \|_{F}^2 =  \left \| Q_{1}^{\top}S^{\top}SB^{\perp}-Q_{1}^{\top}B^{\perp } \right \|_{F}^{2}\le \frac{\epsilon}{4} \left \| B^{\perp } \right \|^{2}_{F} \leq \frac{\epsilon}{4} \left \| \tilde{B}^{\perp } \right \|^{2}_{F} $$
and
$$ \left \|\tilde{P}^{\top}\tilde{S}^{\top}\tilde{S}\tilde{P} - I_s \right\|_{F}^2 = \left\|\tilde{P}^{\top}\tilde{S}^{\top}\tilde{S}\tilde{P} -\tilde{P}^{\top}\tilde{P} \right \|_{F}^2 = \left \| Q_{1}^{\top}S^{\top}S Q_1 -Q_{1}^{\top}Q_1 \right \|_{F}^{2}\le \frac{1}{4} . $$
According to the proof of Theorem \ref{thm4}, the inequality 
$$\left \| \tilde{P}\tilde{Y}_{opt}-\tilde{B}\right \|_{F}^{2}\leq \left(1+\epsilon \right)\left \| \tilde{P}Y_{opt}-\tilde{B}\right \|_{F}^{2}$$ 
holds with probability at least $1- \delta$. Thus,
    $$\frac{1}{2}\left \| P\tilde{X}_{opt}-B\right \|_{F}^{2}+\frac{\sigma}{2}\left\|\tilde{X}_{opt}-C\right\|_{F}^{2}\leq \left(1+\epsilon \right)\cdot OPT $$
holds with probability at least $1- \delta$. 
\end{proof}

\begin{corollary}\label{cor3}
If $m= \max{\left\{ 8s^{2}(2+3^{q})/\delta, \ 8s(2+3^{q})/(\epsilon \delta)\right\}}$, there is at least $1-\delta$ probability that ${\cal G}_k^{(t+1)}$ computed by (\ref{ite}) is a solution with a relative error of $\epsilon$ from the optimal value of (\ref{lrtt_sub_prox}).
\end{corollary}

\section{Numerical Experiment}
\label{sec_5}
    To test the effectiveness of our proposed algorithm (denoted by TT-TS), we compare it with other two algorithms. The first one is the deterministic algorithm TT-ALS \cite{HRS12}, which serves as the baseline for low-rank tensor train decomposition. The second one is the randomized algorithm called TT-Random \cite{WCC21} where the sketching matrix in (\ref{lrtt_sub4}) is chosen such that the rows of $H_k$ are chosen randomly. In the three algorithms, TT-cores are updated from left to right, and we developed our own implementation tailored to the specific problem. All three algorithms take the same TT-ranks as input (the boundary ranks are set to 1). To ensure fairness, we used third-order zero tensors as the initial core tensors for the experiments, and the stopping criteria is either the maximum number of iteration is achieved or the algorithm reaches the tolerance error. The accuracy evaluation for the algorithms is the maximum relative error of TT-cores between two subsequent iterations, calculated using the following formula: 
$$\max_{k=1,2,\ldots, d} \left \{ \frac{\left \| \mathcal{G}_{k}^{(t+1)}-\mathcal{G}_{k}^{(t)} \right \|_{F} }{\left \| \mathcal{G}_{k}^{(t+1)}  \right \|_{F} }  \right \}.$$
    All experiments were conducted using Matlab R2016b on a computer with an AMD E2 7TH-GEN @2.20GHz CPU and 8 GB of RAM. We utilized the MATLAB Tensor Toolbox \cite{BK23} to perform the experiments.

     \subsection{Experimental Results for Synthetic Data}
        In the first synthetic experiment, we randomly generate a sixth-order tensor $\mathcal{A}_{1} \in \mathbb{R}^{10 \times 10 \times \cdots \times 10}$ with TT-format, where the entries of each core are drawn independently from a standard normal distribution. For simplicity, the TT-ranks are equal, i.e., $r_1=r_2=\cdots=r_5$. The true rank of the generated tensor is denoted by $r_{true}$ while the target rank used in the algorithms is denoted by $r$. 
        In addition, the generated tensor has been added by Gaussian noise with standard deviations of 0.1 and 0.01, respectively. The numerical results are reported in Figures \ref{fig_2} and \ref{fig_3}.
\begin{figure}[htbp]
	\centering
	\includegraphics[width=1.0\linewidth]{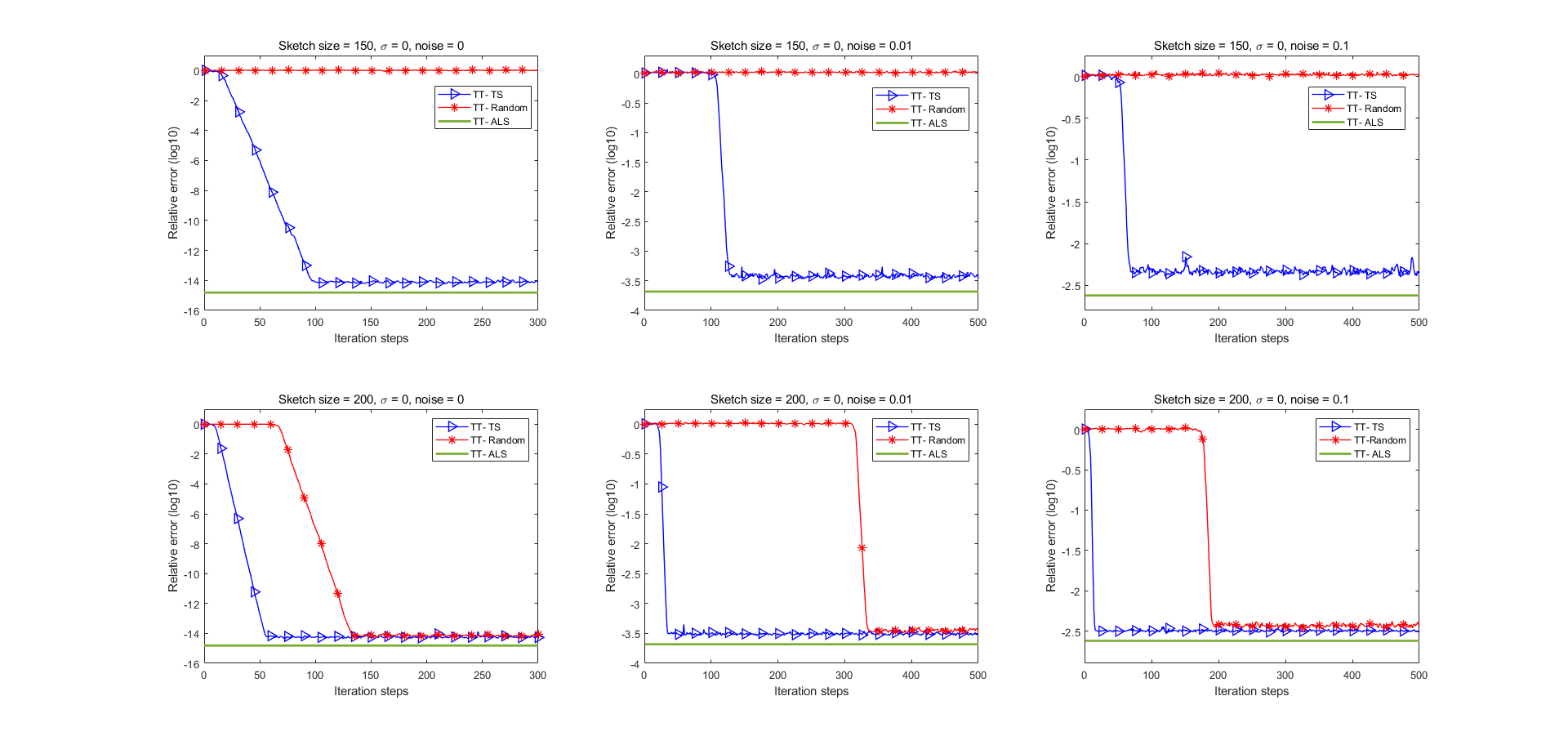}
	\caption{Iteration vs. relative error for tensor $\mathcal{A}_{1}$ with target rank $r=r_{true}= 5$.}
	\label{fig_2}
\end{figure}

\begin{figure}[htbp]
	\centering
	\includegraphics[width=1.0\linewidth]{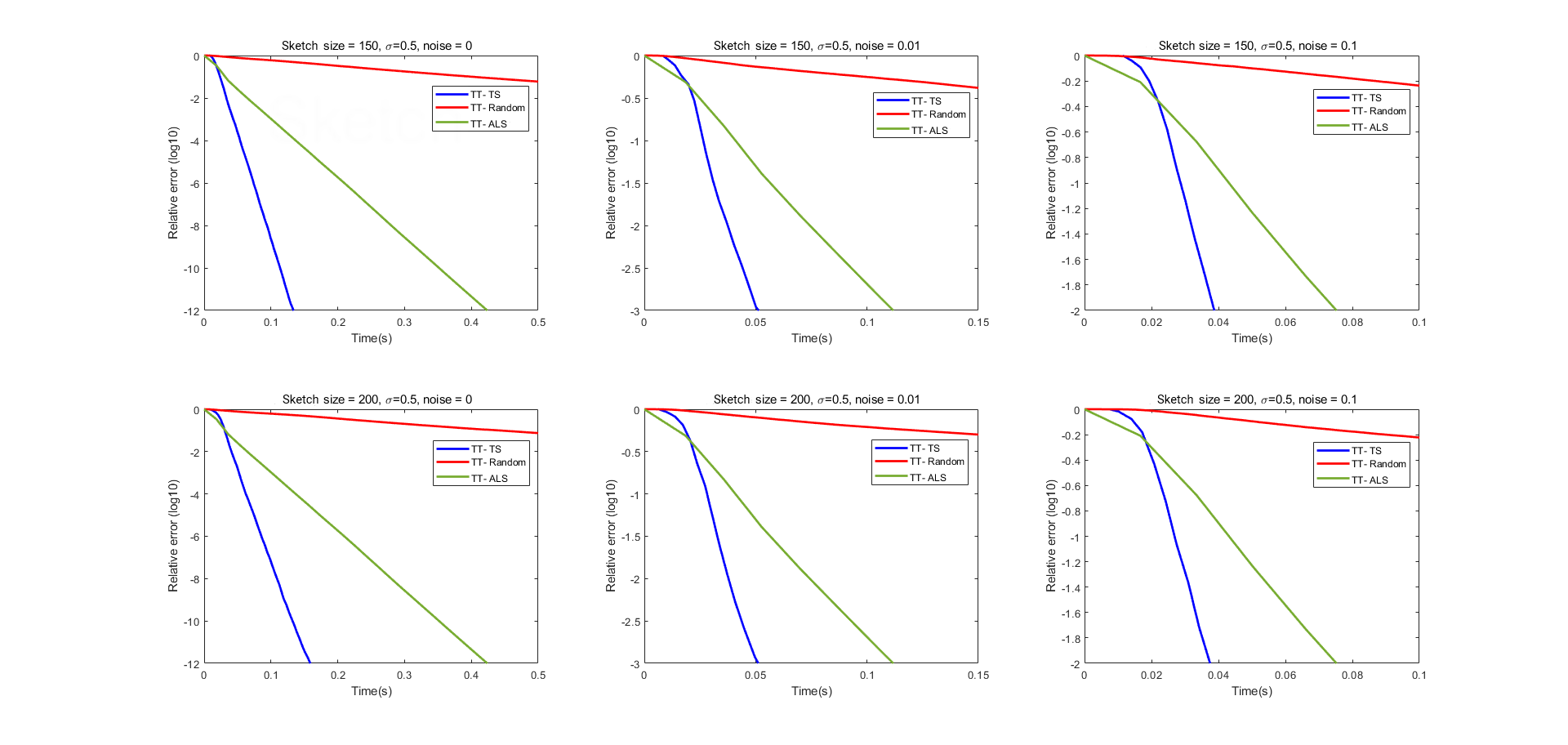}
	\caption{Time vs. relative error for tensor $\mathcal{A}_{1}$ with target rank $r = r_{true}= 5$.}
	\label{fig_3}
\end{figure}
         Figure \ref{fig_2} shows the relationship between the number of iterations and the relative error for TT-TS and TT-Random at the same sketch size, and all the data in the figures are the mean of 10 runs. As we can see, when the sketch size is low (sketch size = 150), our method (TT-TS) requires on average only 100 iterations to achieve an accuracy close to that of TT-ALS, whereas TT-Random fails to converge or shows little improvement in terms of error reduction. When the sketch size is increased to 200, TT-Random significantly reduces the error after about 85 iterations until it reaches an accuracy similar to that of our method, but it takes on average almost 100 more iterations than our method. Furthermore, in the presence of Gaussian noise, TT-Random requires more iterations than our method to achieve the same accuracy. In contrast, our method requires only a small number of samples to achieve an accuracy close to that of TT-ALS. In addition, to avoid the singularity of subproblems, we add a regularization term to the subproblems. Figure \ref{fig_3} shows the relationship between time and relative error for the three algorithms with a proximal term parameter of 0.5. Combining the numerical results in Figures \ref{fig_2} and \ref{fig_3}, we can see that TT-TS requires the least amount of time to compute the TT decomposition of a large-scale tensor. This is because the complexity of our method is much lower than that of TT-ALS, while its accuracy is much higher than that of TT-Random.

     \subsection{Experimental Results for One-dimensional Functions}
        In the second experiment, we use TT-Random and TT-TS to approximate two one-dimensional functions. The first function is $y=sinc(x)$, which is widely used in the fields of signal processing and image processing. The second function is $y=sin(\frac{4}{x}) cos (x^2)$, which is chosen from the highly oscillatory functions considered in \cite{MB21}. We evaluated these two functions at $10^{6}$ points within the intervals $\left[-5,5\right]$ and $\left[0,1\right]$, respectively. Then, we used the command \texttt{reshape} in MATLAB to transform the function values within the intervals into sixth-order tensors, denoted as $\mathcal{A}_{2}, \mathcal{A}_{3} \in \mathbb{R}^{10 \times 10 \times \cdots \times 10}$, respectively. During the approximation process, we set the target rank $r=5$ and sketch size $m=45$ for $\mathcal{A}_{2}$, and set the target rank $r=20$ and sketch size $m=1000$ for $\mathcal{A}_{3}$. We considered the impact of proximal term in the experiments with different values of $\sigma$, i.e., $\sigma =0$ and $\sigma=0.5$. The numerical results are shown in Figures \ref{fig_4} and \ref{fig_5}, which demonstrate the accuracy of the approximation using TT-TS and TT-Random after 100 iterations. The accuracy is measured by computing the relative error (denoted by ``err" in the figures) between the original tensor and the approximate tensor. Additionally, Figure \ref{fig_6} shows the relationship between the time and relative error for the three algorithms with $\sigma =0.5$.

\begin{figure}[htbp]
	\centering
	\includegraphics[width=1.0\linewidth]{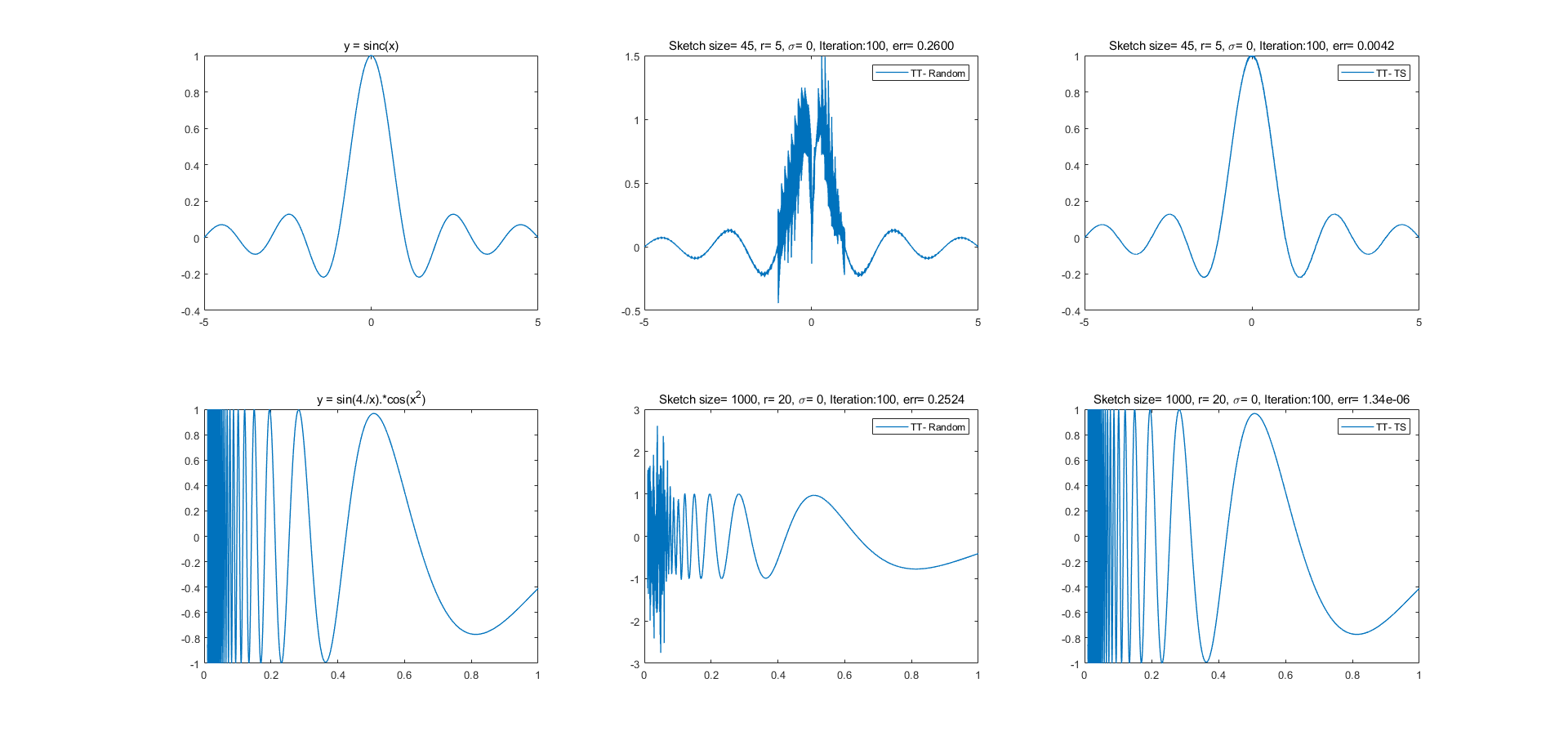}
	\caption{One-dimensional function approximation for $\sigma = 0$.}
	\label{fig_4}
\end{figure}

\begin{figure}[htbp]
	\centering
	\includegraphics[width=1.0\linewidth]{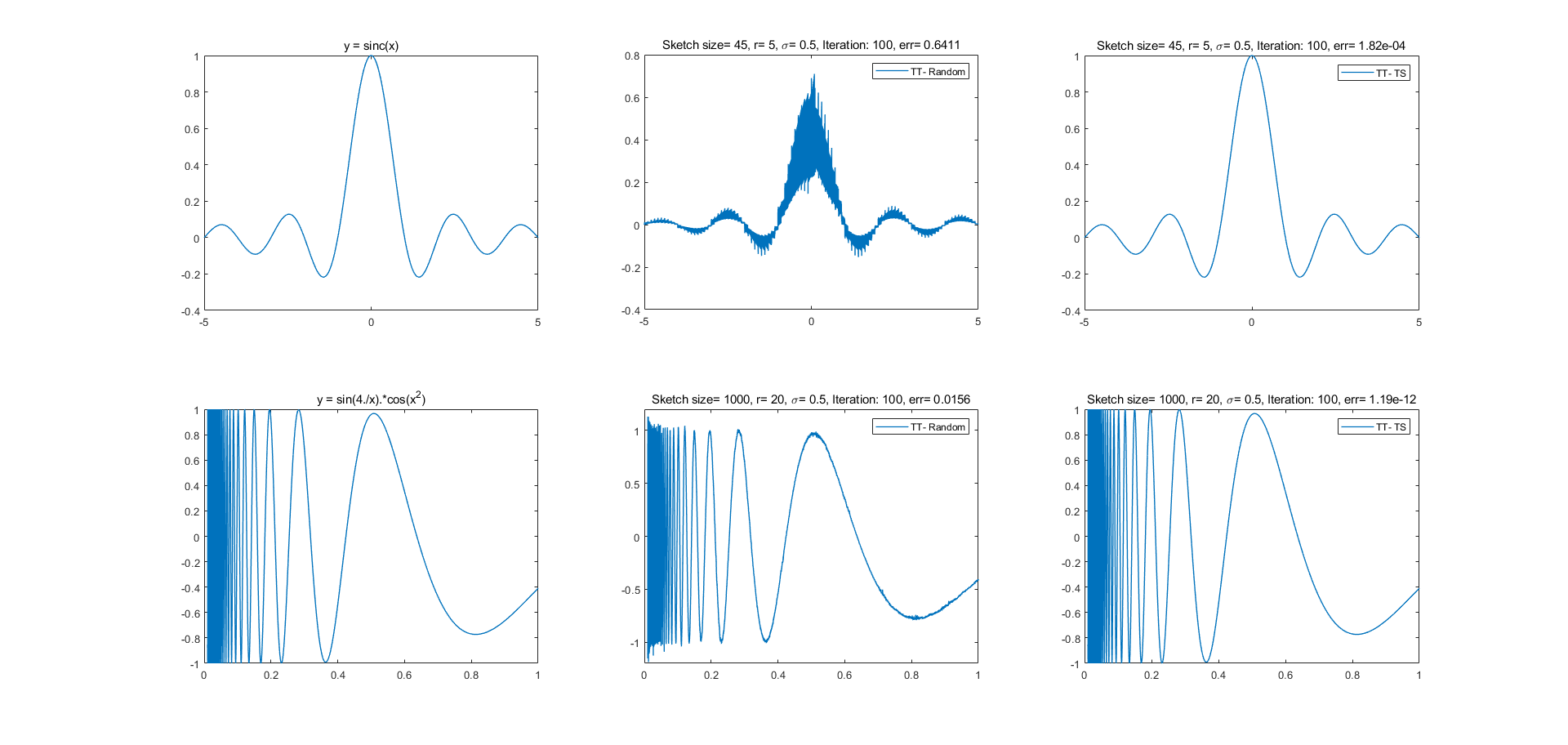}
	\caption{One-dimensional function approximation for $\sigma = 0.5$.}
	\label{fig_5}
\end{figure}

\begin{figure}[htbp]
	\centering
	\includegraphics[width=1.0\linewidth]{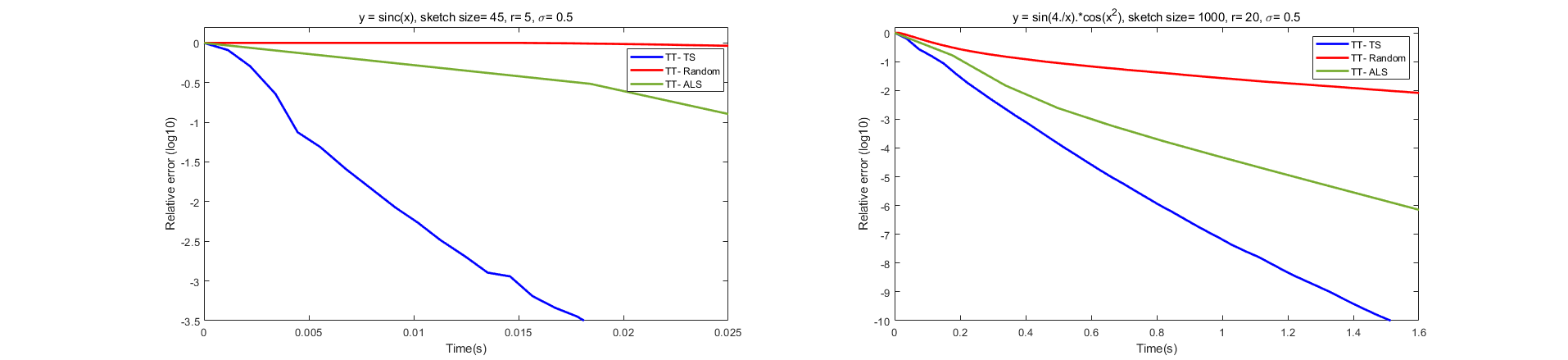}
	\caption{Time vs. relative error for $\mathcal{A}_{2}$ and $\mathcal{A}_{3}$ with $\sigma = 0.5$.}
	\label{fig_6}
\end{figure}

        The experimental results show that the approximations obtained from the TT-TS algorithm are more accurate than that obtained from the TT-Random algorithm at the same number of iterations. In addition, the accuracy of both algorithms is significantly improved after adding the proximal term. To sum up, under the same experimental conditions, our method outperforms the TT-Random algorithm. In terms of time cost, the TT-TS algorithm takes the least amount of time to achieve an accuracy comparable to that of the TT-ALS algorithm. Thus we can conclude that our method is fast and efficient for the low-rank TT approximation of one-dimensional functions.

\subsection{Experimental Results for Real Data} 
        In this section, we consider three real datasets consisting of image, hyperspectral and video data. Next, we provide a brief overview of the data for the experiments, which is summerized in Table \ref{tab2}.
    \begin{itemize}
         \item The first data is $pompoms$, which is an RGB colour image dataset $(512\times512\times3)$ derived from the CAVE databases\footnote{\href{https://www1.cs.columbia.edu/CAVE/databases/multispectral/stuff/}{https://www1.cs.columbia.edu/CAVE/databases/multispectral/stuff/}}, where 512 represents the height and width of the image in pixels, and 3 represents the three color channels (red, green and blue) that make up each pixel.
         \item The second data is $Indian\_pines$, which is a hyperspectral image dataset $(145\times145\times220)$ sourced from Hyperspectral Remote Sensing Scenes\footnote{\href{https://www.ehu.eus/ccwintco/index.php?title=Hyperspectral_Remote_Sensing_Scenes}{https://www.ehu.eus/ccwintco/index.php?title=Hyperspectral\_Remote\_Sensing\_Scenes}}. It is a third-order tensor containing hyperspectral images, where the first two dimensions represent the height and width of the image, and the third dimension represents the number of spectral bands.
         \item The third data is $Skate$, which is a video dataset $(720\times1280\times3\times420)$ sourced from Pixabay\footnote{\href{https://pixabay.com/videos/skate-sport-water-action-exercise-110734/}{https://pixabay.com/videos/skate-sport-water-action-exercise-110734/}}. It is a fourth-order tensor representing color video of a man surfing on the sea. The various dimensions of this tensor represent different aspects of the video data, including its resolution, color space, and frame rate. Here, we selected the information from the first 30 frames.
    \end{itemize}
\begin{table}[htbp]
\caption{Size and type of real data.} \label{tab2}
\centering
\begin{tabular}{ccc}
\hline\hline
Data & Size & Type \\ \hline\hline
$pompoms$ & $512 \times 512 \times 3$ & RGB Image \\ \hline
$Indian\_ pines$ & $145 \times 145 \times 220$ & Hyperspectral Image \\ \hline
$Skate$ & $720 \times 1280 \times 3 \times 30$ & Video \\ \hline
\end{tabular}
\end{table}

For the original image tensor ${\cal A} \in \mathbb{R}^{n_1 \times n_2 \times n_3}$, the quality of the approximate tensor ${\tilde{\cal A}}$ is measured by the peak signal-to-noise ratio (PSNR) which is defined as 
$$ \text{PSNR} = \frac{1}{n_3} \sum_{i_3 =1}^{n_3} 10 \cdot \log_{10} \frac{255^2}{\left\| \mathcal{A}(:,:,i_3)- \tilde{\cal A}(:,:,i_3) \right\|_F^2}. $$
It is reasonable that the deterministic algorithm TT-ALS gives the best quality of approximation since all the information of subproblems is used. However, the computation time of each sweep is much higher than that of randomized algorithms TT-TS and TT-Random, especially for large-scale tensors as shown in Table \ref{tab2}. Here we mainly compare the numerical results of TT-TS and TT-Random, and use the results generated by TT-ALS as baselines. 

For the colour image data, we set the TT-ranks as $\left(1, 50, 3, 1\right)$ and the experimental results are shown in Figure \ref{fig_7}. In addition, we also compare the effect of sketch size and TT-ranks on the experiment as shown in Figure \ref{fig_8}. From Figure \ref{fig_7}, we can clearly see that for TT-TS and TT-Random, the approximations would be more accurate as the sketch size increases. Under the same settings, the experimental results of the TT-TS are always superior to that of the TT-Random. For example, when the sketch size is 200, the PSNR of TT-TS is 30.63 while the PSNR of TT-Random is 29.7. From Figure \ref{fig_8}, we can see that when TT-ranks are set as $\left(1, 100, 3, 1\right)$, TT-TS still outperforms TT-Random at the same sketch size. For both cases, the gap of PSNR between TT-TS and TT-ALS gets smaller and smaller as the sketch size increases.
        

\begin{figure}[htbp]
	\centering	
        \includegraphics[width=1.0\linewidth]{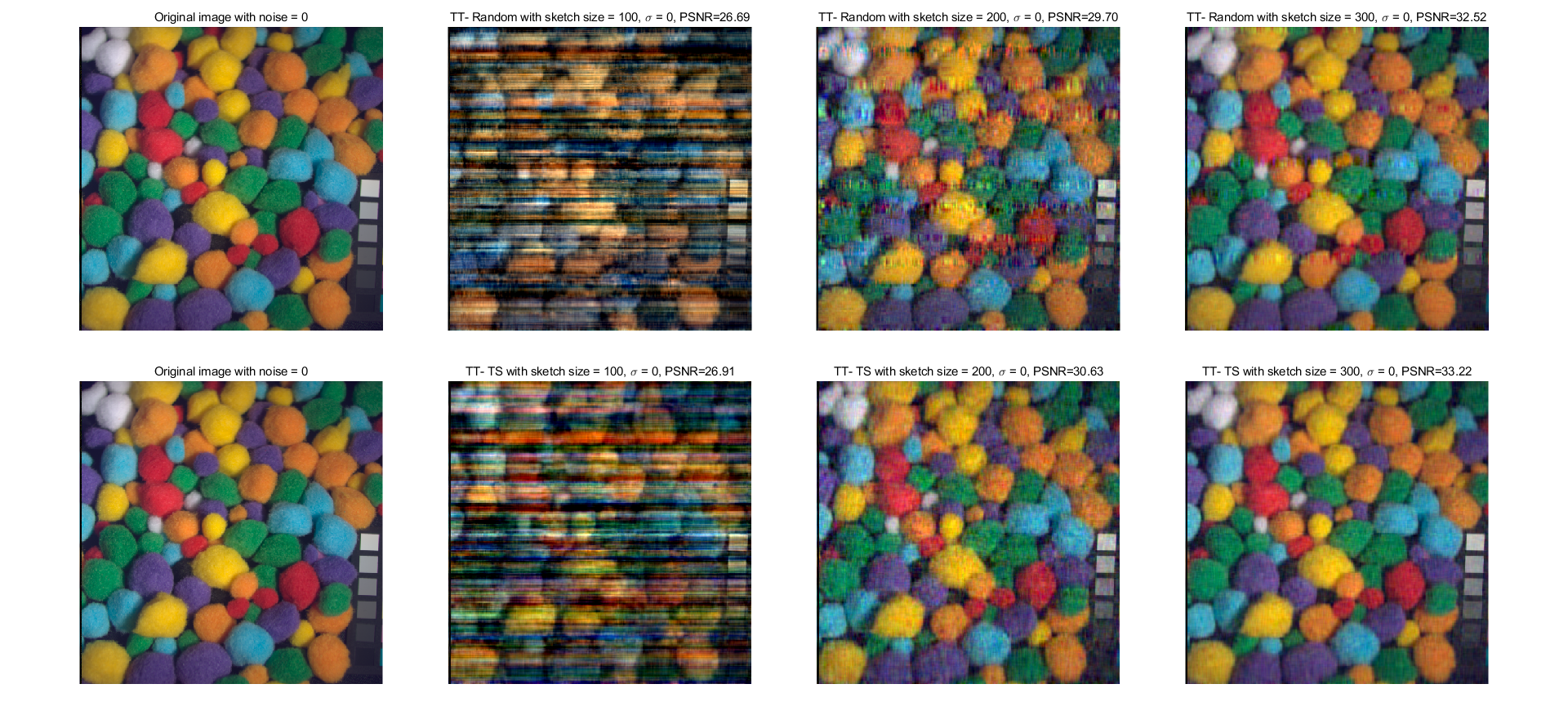}
	\caption{Numerical results for color image with TT-ranks $r=\left(1,50,3,1\right)$.}
	\label{fig_7}
\end{figure}

\begin{figure}[htbp]
	\centering
	\includegraphics[width=1.0\linewidth]{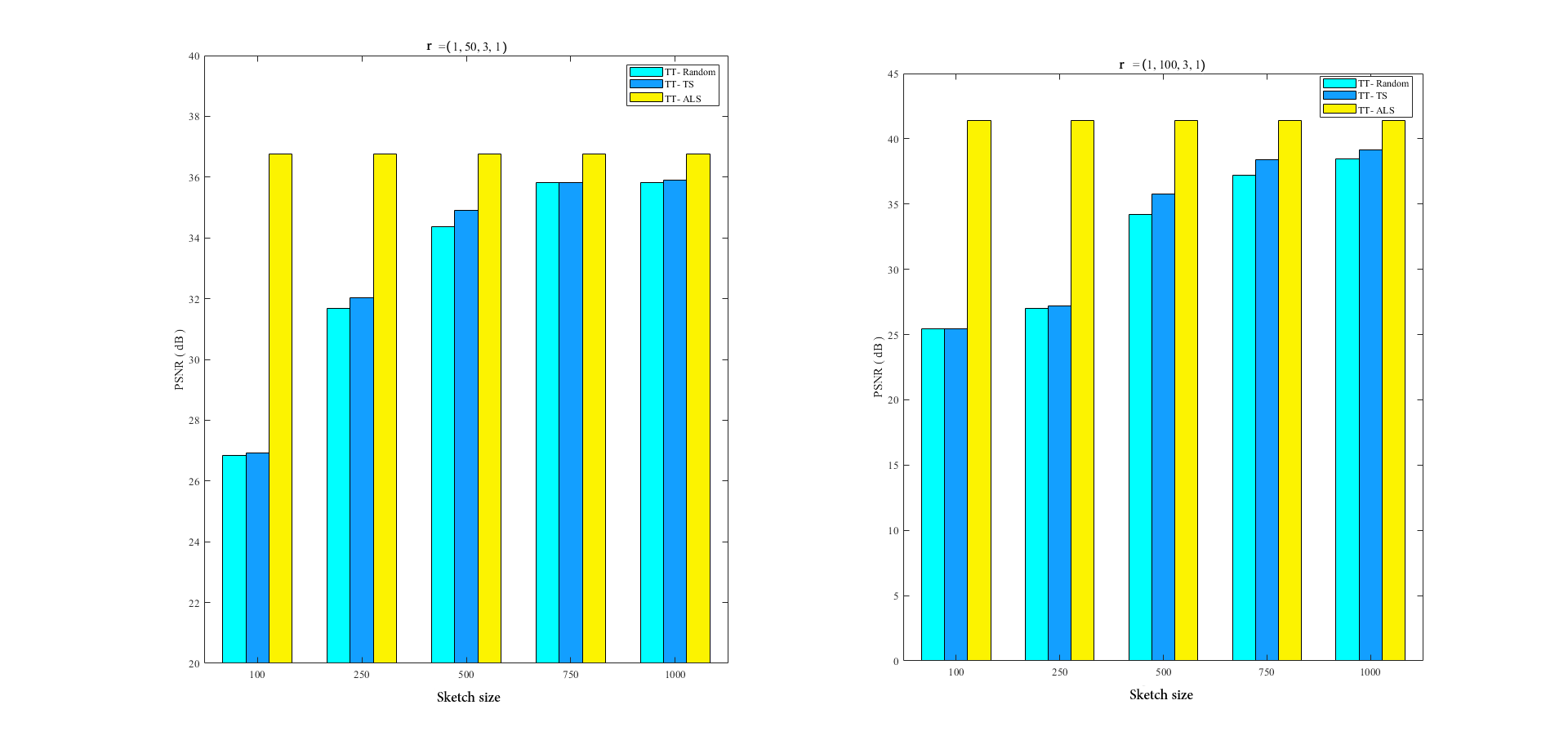}
	\caption{PSNR vs. sketch size for color image with different TT-ranks.}
	\label{fig_8}
\end{figure}

For hyperspectral image data, we set the TT-ranks as $\left(1, 20, 20, 1\right)$ and the experimental results are shown in Figure \ref{fig_9}. We also consider the effect of the three algorithms on the approximation of the original data when the proximal parameter $\sigma = 0.5$, and the numerical results are shown in Figure \ref{fig_10}. In the experiments, we draw the spectral curve of the hyperspectral image for pixel at position $(1, 1)$. The accuracy of approximation is measured by the relative error (denoted by ``err" in the figures) between the original spectral curve and the approximate spectral curve.
As can be seen from Figures \ref{fig_9} and \ref{fig_10}, the approximation of TT-TS is always better than that of TT-Random under the same settings (the values of sketch size, proximal parameter and iteration number). As the sketch size increases, both the approximations of TT-TS and TT-Random become more accurate and the gap between TT-TS and TT-ALS gets smaller and smaller. Besides, both the approximation accuracies of TT-TS and TT-Random are improved by adding the proximal term, which indicates the significance of regularization.

\begin{figure}[htbp]
	\centering
	\includegraphics[width=1.0\linewidth]{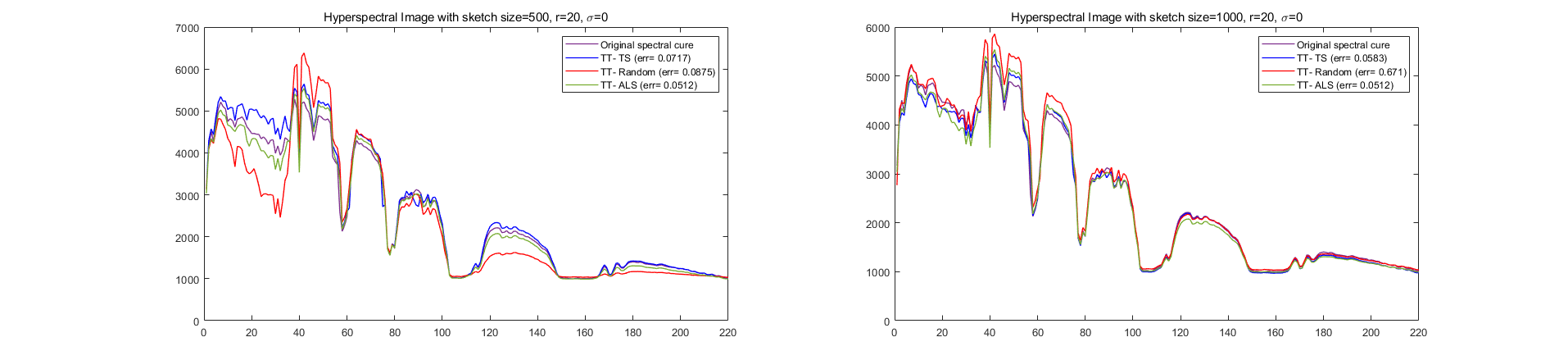}
	\caption{Numerical results for hyperspectral image with $\sigma = 0$.}
	\label{fig_9}
\end{figure}

\begin{figure}[htbp]
	\centering
	\includegraphics[width=1.0\linewidth]{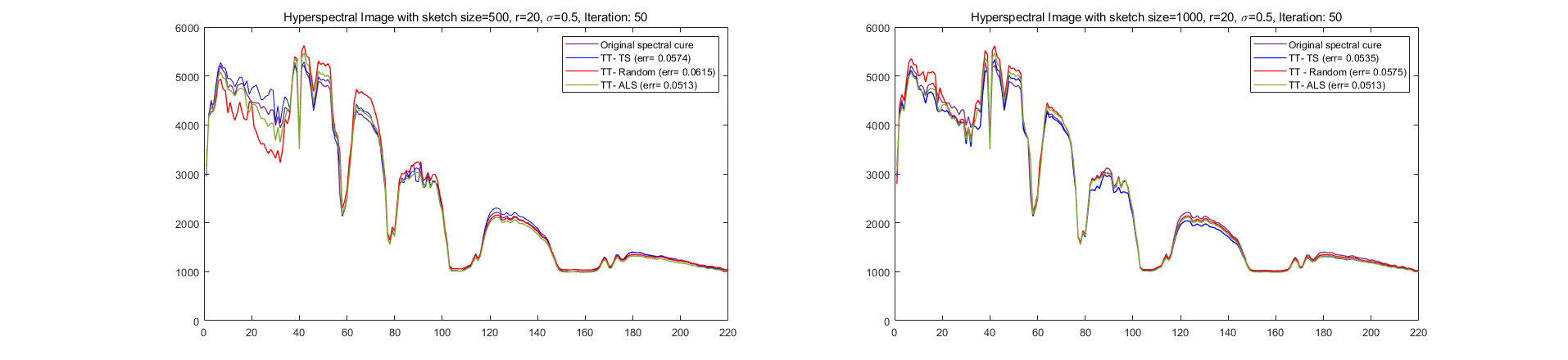}
	\caption{Numerical results for hyperspectral image with $\sigma = 0.5$.}
	\label{fig_10}
\end{figure}

For the video data, we set the TT-ranks as $\left(1, 20, 20, 20, 1\right)$ and $\sigma =0.5$. To verify the efficiency of randomized algorithms, we record the average computation time of each sweep for the tensors generated by the first 10 and 30 frames of the video, respectively. The numerical results of the three algorithms are presented in Table \ref{tab3}. We also compare the approximation results of TT-TS and TT-Random on the first 30 frames of video data under the same sketch size and iteration number. For visualization purpose, we only present the approximation results of the second frame of the video as shown in Figure \ref{fig_11}. From Table \ref{tab3}, we can see that TT-TS and TT-Random take much less time than TT-ALS. As the size of data increases, TT-ALS may run out of memory while the two randomized algorithms are still able to work. According to Figure \ref{fig_11}, the PSNR of TT-TS is 35.23 and the PSNR of TT-Random is 25.1 when $\sigma =0.5$, which demonstrates the superiority of TT-TS over TT-Random under the same conditions.

\begin{table}[htbp]
\footnotesize
\caption{The average computation time of each sweep for video dataset with TT-ranks $r=\left(1,20,20,20,1\right)$ and $\sigma =0.5$.} \label{tab3}
\begin{center}
\begin{tabular}{c|ccccc|ccccc}
\hline 
Data size & \multicolumn{5}{c|}{$720 \times 1280 \times 3 \times 10$} & \multicolumn{5}{c}{$720 \times 1280 \times 3 \times 30$} \\ \hline
Methods & TT-ALS & \multicolumn{2}{c}{TT-Random} & \multicolumn{2}{c|}{TT-TS} & TT-ALS & \multicolumn{2}{c}{TT-Random} & \multicolumn{2}{c}{TT-TS}\\ \hline
Sketch size & All & 1000 & 2000 & 1000 & 2000 & All & 1000 & 2000 & 1000 & 2000 \\ \hline 
Time(s) & 25.67 & 0.07 & 0.12 & 0.04 & 0.06 & out of memory & 0.07 & 0.13 & 0.05 & 0.07\\ \hline
\end{tabular}
\end{center}
\end{table}

\begin{figure}[htbp]
	\centering
	\includegraphics[width=1.0\linewidth]{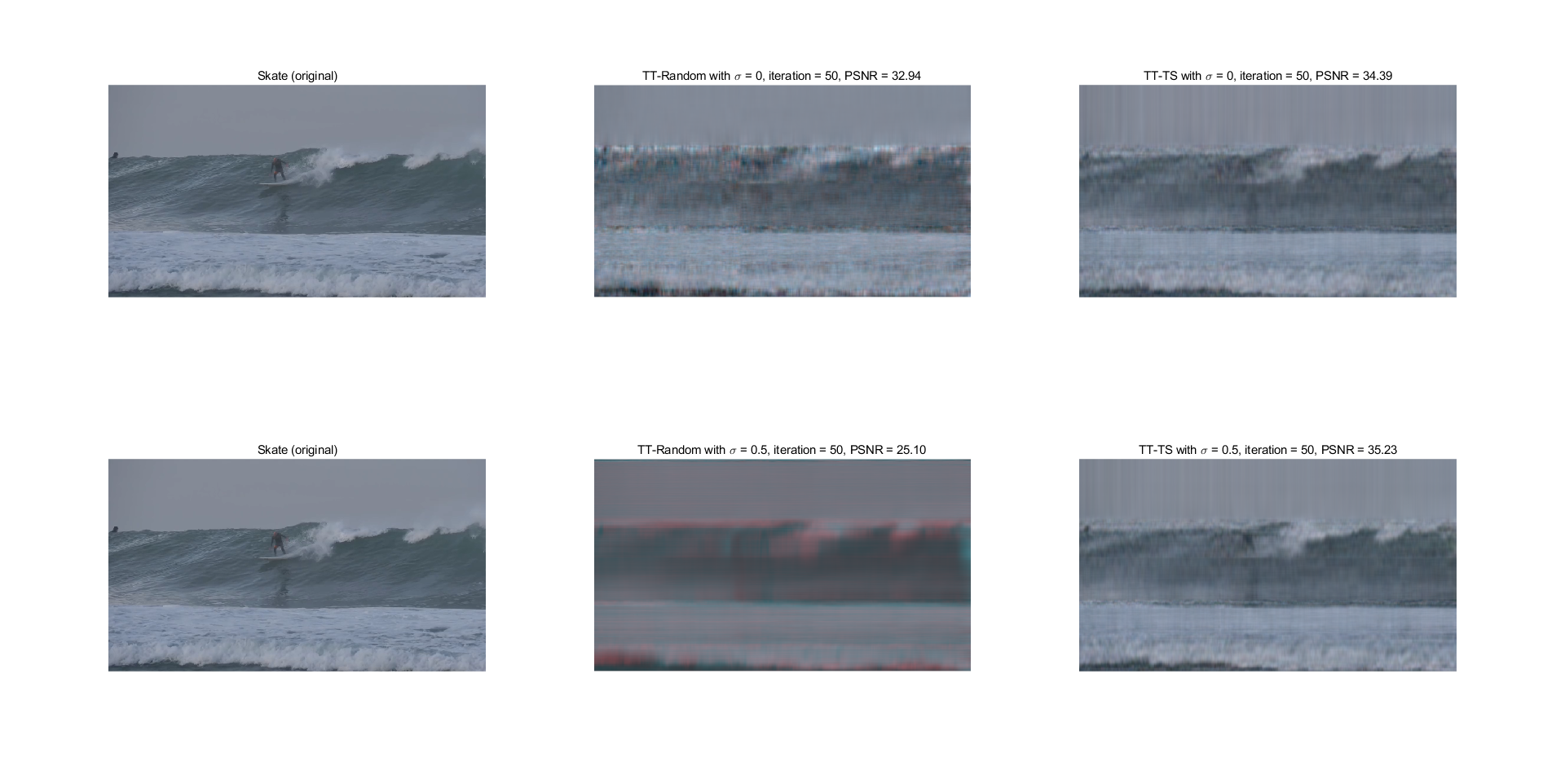}
	\caption{Numerical results for video experiment with sketch size $m=1000$ and TT-ranks $r=\left(1,20,20,20,1\right)$.}
	\label{fig_11}
\end{figure}

\section{Conclusion}
\label{sec_6}
In this paper, we proposed a novel randomized proximal ALS algorithm for low-rank tensor train decomposition by using TensorSketch. The fast computation and accuracy of TensorSketch make our algorithm more practical for computing the low-rank TT decomposition of large-scale tensors. Numerous experiments on both synthetic and real datasets were conducted to demonstrate the effectiveness and efficiency of the proposed algorithm. The numerical results showed the superiority of our algorithm in terms of computation complexity and accuracy for low-rank TT decomposition. On the other hand, we found that the theoretical lower bounds of sketch size are too conservative for our randomized algorithm. Further research on the estimate of sketch size for TensorSketch is needed.


\bibliographystyle{siamplain}
\bibliography{TT-TS}

\end{document}